\documentclass[11pt]{article}

\usepackage[all]{xy}
\usepackage{tikz-cd}
\usepackage{color}

\usepackage{amssymb,latexsym,times}
\usepackage{amsthm}
\usepackage{hyperref} 

\newtheorem{Thm}{Theorem}%[section]
\newtheorem{Lem}[Thm]{Lemma}
\newtheorem{Cor}[Thm]{Corollary}
\newtheorem{Prop}[Thm]{Proposition}

\newtheorem{example}[Thm]{Example}
\newtheorem{remark}[Thm]{Remark}
\newtheorem{Def}[Thm]{Definition}

\def\resp{{\em resp.\ }}
\newcommand{\sslash}{\mathbin{/\mkern-6mu/}}
\newcommand{\wt}{\widetilde}

\newcommand{\red}{\color{red}}

\newcommand{\green}{\color{green}}

\begin{document}

%------
% Insert the title of your paper and (if necessary)
% a short title for the running head.
%------

\title{\bf An introduction to $(G,c)$-bands}
\author{L. Francone and B. Leclerc}

\date{}

%%%%%%%%%%
\maketitle
%%%%%%%%%%

\begin{abstract}
We give an introduction to our results on cluster structures for schemes of $(G,c)$-bands \cite{FL},
emphasizing their connections with seminal works of Frenkel and Reshetikhin in the 90's.
In particular, we construct using $(G,c)$-bands a discrete analogue of the $q$-difference Miura 
transformation of the loop group $LG$ \cite{FR0}, and we show that it calculates the $q$-characters 
of the finite-dimensional representations of the quantum affine algebra $U_q(\widehat{\mathfrak{g}})$ 
of the same $A$, $D$, $E$ type as $G$, thus verifying a conjecture of Frenkel and Reshetikhin \cite{FR}.
\end{abstract}

\setcounter{tocdepth}{3}
{\small \tableofcontents}

\section{Introduction}

In \cite{FL} we introduced the notion of a $(G,c)$-band associated 
with a simple and simply connected algebraic group $G$ of Cartan type $A$, $D$, $E$, and a Coxeter
element $c$ in its Weyl group. We proved that $(G,c)$-bands are the rational points of
an infinite-dimensional affine scheme $B(G,c)$ whose ring  of regular functions
$R(G,c)$ has the structure of a cluster algebra. This provided a geometric model for
the recently discovered cluster algebra structure in the Grothendieck ring of a certain
category of representations of shifted quantum affine algebras of the same Lie type as $G$ \cite{GHL}.

The scheme $B(G,c)$ has a natural $G$-action, and one can consider invariant subalgebras
of $R(G,c)$ for the induced action of various subgroups of $G$. We showed
that the invariant subalgebras $R(G,c)^U$ (where $U$ denotes a maximal unipotent subgroup of $G$)
and $R(G,c)^G$ are cluster subalgebras, also arising in connection with the representation
theory of quantum affine algebras. 

In fact, $R(G,c)^G$ is isomorphic to the Grothendieck ring
of a monoidal category $\mathcal{C}_\mathbb{Z}$ of finite dimensional representations of 
the quantum affine algebra $U_q(\widehat{\mathfrak{g}})$ of the same Cartan type as $G$.
There is a huge literature on these finite-dimensional representations because of their numerous connections with mathematical physics and integrable systems. 

In the 90's Frenkel and Reshetikhin \cite{FR} introduced the $q$-character $\chi_q(M)$
of a finite-dimen\-sio\-nal $U_q(\widehat{\mathfrak{g}})$-module $M$, which plays a similar role as the ordinary character $\chi(V)$ of a finite-dimensional $G$-module $V$. 
The map $M\mapsto \chi_q(M)$ induces an injective homomorphism from the Grothendieck ring to
a Laurent polynomial ring.
By construction, $\chi_q(M)$ is the generating function of the dimensions 
of the common generalized eigenspaces of $M$ under the action of a certain subalgebra of 
$U_q(\widehat{\mathfrak{g}})$ similar to the Cartan subalgebra of $\mathfrak{g}$.
However, a second more intriguing description of the $q$-character homomorphism is discussed in \cite[\S8]{FR}, involving $q$-difference analogues of the Drinfeld-Sokolov reduction and of the Miura transformation. 
As explained in \cite{FRS,SS} the $q$-difference Drinfeld-Sokolov reduction
relies on an analogue for the formal loop group $LG$ of a classical theorem of Steinberg for~$G$
\cite{S}. This theorem describes a cross-section of the set of regular conjugacy classes in~$G$, 
whose coordinate ring is naturally identified with the character ring of $G$.

It turns out that Steinberg's cross-section is also the starting point of the definition 
of a $(G,c)$-band. Moreover, a discrete analogue of the $q$-difference Miura transformation of $LG$ naturally arises in the context of $(G,c)$-bands, yielding a geometric interpretation of the $q$-character homomorphism. 
This prompted us to give an introduction to the constructions and results of \cite{FL} by taking a historical perspective and emphasizing their similarities with the inspiring ideas of the classical papers \cite{S},\cite{FR}.

In \S\ref{sect1}, we start by recalling Steinberg's cross-section theorem and its relation with
characters of $G$-modules. In \S\ref{sect2}, we review its loop analogue obtained in \cite{FRS,SS}.
We also discuss the $q$-difference Miura transformation and its conjectural relation with $q$-characters of
$U_q(\widehat{\mathfrak{g}})$-modules observed by Frenkel and Reshetikhin. After this preparation we begin our exposition on $(G,c)$-bands. In \S\ref{sect3}, we explain the definition of a $(G,c)$-band and the action of $G$
on the affine scheme $B(G,c)$. 
In \S\ref{sect-invariant-subalg}, we present our results 
on the invariant subalgebras $R(G,c)^G$ and $R(G,c)^U$, as well as their relations with the categories $\mathcal{C}_\mathbb{Z}$ and $O^+_\mathbb{Z}$. 
In \S\ref{sect-discrete-Miura}, we show that the $q$-character of a finite-dimensional 
$U_q(\widehat{\mathfrak{g}})$-module can be regarded as a cluster expansion with respect to the 
distinguished cluster of $R(G,c)^U$ (or $R(G,c)^{U^-}$). This follows from the generalized Baxter's relations of
\cite{FH}. We then construct a discrete analogue $H$ of the $q$-difference Miura transformation, and 
we show that the $q$-character homomorphism can be identified with the pullback $H^*$ of this morphism $H$ (Theorem~\ref{Thm-FR}). This verifies a conjecture of Frenkel and Reshitikhin for all types $A, D, E$ and all Coxeter elements~$c$.
We also give a discrete analogue of the cross-section theorem of Frenkel, Reshetikhin, Semenov-Tian-Shansky and Sevostyanov (Theorem~\ref{thm:cross-sec discrete}), and use it to prove Proposition~\ref{prop7}, which states that the classical characters
$Q^{(i)}_k$ of the Kirillov-Reshetikhin modules have a natural interpretation as special coordinate functions on Steinberg's 
cross-section in $G$. We find it remarkable that one can discover 
in a subvariety of a purely classical algebraic group, like Steinberg's cross-section, the shadows of key objects of the quantum affine representation theory, like the Kirillov-Reshetikhin modules. 
Finally, in \S\ref{sect4}, we describe the cluster structure of $R(G,c)$ and its connections
with Hernandez' category $\mathcal{O}^{\rm shift}$ of representations of shifted quantum affine algebras \cite{H}. 

The proofs of most of our results can be found in \cite{FL} and are therefore omitted. 
Notable exceptions are Proposition~\ref{prop6}, Proposition~\ref{prop7}, Theorem~\ref{Thm-FR}, and Theorem~\ref{thm:cross-sec discrete}, 
which are new, and for which we include complete proofs.

%%%%%%%%%%%%%%%%%%%%%%%
\section{Steinberg's cross-section theorem}\label{sect1}

\subsection{Conjugacy classes of regular elements and the character ring}\label{subsec1.1}

Let $G$ be a semisimple algebraic group of rank $r$ over $\mathbb{C}$. 
Following Steinberg \cite{S}, an element $g$ of $G$ is called \emph{regular} if its centralizer in $G$ has dimension $r$.
Equivalently, a regular element is one whose centralizer has the least possible dimension,
or whose conjugacy class has the greatest possible dimension. Note that regular elements are not assumed to be semisimple. 

Let $T$ be a maximal torus of $G$, $B$ a Borel subgroup containing $T$, $U$ its unipotent radical. Let $B^-$ be the Borel subgroup opposite to $B$ with respect to $T$, $U^-$ its unipotent radical. Let $W = N(T)/T$ be the Weyl group, with Coxeter generators $s_i\ (1\le i\le r)$ corresponding to the simple roots relative to $B$. Let $c$ be the product of the $s_i$'s in a given order (a Coxeter element of $W$), and $\overline{c}$ a representative of $c$ in $N(T)$. We define
\[
 A := U(c^{-1}) \overline{c},
\]
where $U(c^{-1})$ denotes the $r$-dimensional unipotent subgroup $U \cap (cU^-c^{-1})$. Then $A$ is an affine subspace of $G$ of dimension $r$.

In the sequel, we further assume that $G$ is simply connected. Then the fundamental representations of the Lie algebra of $G$ give rise to the fundamental irreducible $G$-modules. 
Let $\chi_i \in \mathbb{C}[G]\ (1\le i \le r)$ be the characters of these fundamental $G$-modules.

Steinberg proves the following theorem:

\begin{Thm}[\cite{S}, Theorem 1.4]\label{thm1}
\begin{enumerate}
\item The space $A$ is a cross-section of the collection of conjugacy classes of regular elements of $G$.
\item The map $\chi : G \to \mathbb{C}^r$ defined by 
\[
 \chi(g) = (\chi_1(g),\ldots, \chi_r(g))
\]
restricts to an isomorphism from $A$ onto $\mathbb{C}^r$.
\end{enumerate}
\end{Thm}

\begin{example}\label{exa1}
{\rm
Let $G = SL(n)$, of rank $r=n-1$, with its subgroup $B$ of upper triangular matrices. 
Let $c = s_1s_2\cdots s_{r}$ be the standard Coxeter element, so that
\[
\overline{c} =
\pmatrix{
 0 & 0  &\cdots & 0 &(-1)^{r}\cr
 1 & 0  & \cdots & 0 &0\cr
 0 & 1  & \cdots & 0& 0\cr 
 \vdots&\vdots&\ddots&\vdots & \vdots\cr
 0 & 0 & \cdots & 1& 0
}.
\]
It is easy to check that the subgroup $U(c^{-1})$ consists of all matrices
of the form
\[
u=
\pmatrix{
 1 & a_1  &\cdots & a_{r}\cr
 0 & 1  & \cdots & 0\cr
 \vdots&\vdots&\ddots& \vdots\cr
 0 & 0 & \cdots & 1
 },
\qquad (a_1,\ldots, a_{r} \in \mathbb{C}),
\]
and that $A$ is the affine $r$-dimensional subspace of $G$ consisting of elements of the form 
\begin{equation}\label{eq1}
a=
\pmatrix{
 a_1 & a_2  &\cdots & a_{r} &(-1)^{r}\cr
 1 & 0  & \cdots & 0 &0\cr
 0 & 1  & \cdots & 0& 0\cr 
 \vdots&\vdots&\ddots&\vdots & \vdots\cr
 0 & 0 & \cdots & 1& 0
},
\qquad (a_1,\ldots, a_{r} \in \mathbb{C}).
\end{equation}
We recognize one of the classical normal forms for a matrix which is regular, in the sense that its 
minimal and characteristic polynomials are equal. 

The fundamental $SL(n)$-modules are realized as the exterior powers $\Lambda^i(\mathbb{C}^n)$ of the 
defining representation $\mathbb{C}^n$, so that for $g\in SL(n)$ and $1\le i \le r$, 
\[
\chi_i(g) = \mathrm{tr}(\Lambda^i(g))
\]
is equal to the sum of the principal minors of size $i$ of $g$. Hence if $a$ is of the form (\ref{eq1}), we 
can readily calculate
\[
\chi_i(a) = (-1)^{i-1} a_i,\qquad (1\le i \le r),
\]
which shows that $\chi_{|A} : A \to \mathbb{C}^r$ is an isomorphism. \qed
}
\end{example}

For a general $G$, we will need to replace the minors of a matrix by the generalized minors
\[
 \Delta_{v(\varpi_i), w(\varpi_i)}\in \mathbb{C}[G],\qquad (1\le i \le r,\ v,w\in W)
\]
in the sense of Fomin and Zelevinsky \cite{FZ}. Here $\varpi_i$ denotes the $i$th fundamental weight.

The affine space $A$ has the natural coordinate functions
\[
 \theta_i := \Delta_{\varpi_i,\varpi_i},\ \qquad (1\le i \le r).
\]
Indeed, if $c=s_{i_1}\cdots s_{i_r}$, every element $a$ of $A$ can be written in a unique way as 
\[
a = x_{i_1}(t_1)\overline{s_{i_1}}\cdots x_{i_r}(t_r)\overline{s_{i_r}},\qquad (t_i\in \mathbb{C}),
\]
where  
$x_i: \mathbb{C} \to U$ is the additive one-parameter subgroup corresponding to the simple root $\alpha_i$,
and $\overline{s_i}$ is a fixed representative of $s_i$ in $N(T)$, suitably normalized as 
in \cite{FZ}. Then one can check that 
\[
\Delta_{\varpi_{i_j},\varpi_{i_j}}(a) = t_j,\qquad (1\le j \le r). 
\]
For example, if $G = SL(n)$ and $a$ is of the form (\ref{eq1}), we have 
\[
 \theta_i(a) = (-1)^{i-1} a_i = \chi_i(a), \ \qquad (1\le i \le r),
\]
so that $\theta_i$ coincides with the restriction of $\chi_i$ to $A$. In general, as explained 
by Steinberg \cite[\S7]{S}, the relation between the functions $\theta_i$ and $\chi_i$ is more subtle, and we will
return to it in \S\ref{subsec_1.3} below.

Steinberg's theorem implies that the coordinate ring $\mathbb{C}[A]$ of the affine subspace $A$ can be identified with the polynomial ring in the fundamental characters $\chi_i$, 
that is, as $G$ is simply connected, with the ring of regular functions on $G$ which are $G$-invariant for the conjugation action \cite[Theorem 6.1]{S}.

\subsection{Relation with the characters of $T$}\label{subsec:1.2}

Rather than expressing the character of a $G$-module $M$ as a polynomial in the $\chi_i$'s, one usually expresses it as a $W$-invariant polynomial in the characters of $T$ encoding the weight space decomposition of $M$.
This is achieved using the classical restriction theorem (see \cite[\S6]{S}):
\begin{Thm}\label{Thm-W-inv}
 The natural map from the ring of $G$-invariant regular functions on $G$ to the ring of $W$-invariant regular functions on $T$ is an isomorphism. 
\end{Thm}

In order to relate this picture with the $q$-difference Drinfeld-Sokolov reduction of \S\ref{sect2} below,
we are going to construct in a less conventional way the homomorphism $\mathbb{C}[A] \to \mathbb{C}[T]$ obtained by combining the discussion at the end of the previous section with Theorem~\ref{Thm-W-inv}. We start from the following proposition (see \cite[\S 8.8]{S}).

\begin{Prop}
The Bruhat cell $BcB$ consists of regular elements of $G$. 
\end{Prop}

This implies that the reduced double Bruhat cell 
\[
L^{c,e} := (U\bar{c}U)\cap B^-
\]
is contained in the 
subset of regular elements. It follows that we have a morphism of algebraic varieties $\psi : L^{c,e} \to A$
mapping an element $g\in L^{c,e}$ to its unique conjugate $\psi(g) \in A$. 

Recall from \cite[\S3]{YZ} that $L^{c,e}$ is biregularly isomorphic to a Zariski open subset of an affine space of dimension $r$. An element $g$ of the double Bruhat cell 
\[
G^{c,e} := (BcB)\cap B^-
\]
belongs to   
$L^{c,e}$ if and only if $\Delta_{c(\varpi_i),\varpi_i}(g) = 1$ for all $i = 1,\ldots, r$.
A natural coordinate system on $L^{c,e}$ is given by the functions 
$\Delta_{\varpi_i,\varpi_i}\ (1\le i\le r)$. 
In fact, by \cite{BFZ}, the coordinate ring of any (reduced) double Bruhat cell has a cluster algebra
structure, and the functions $\Delta_{\varpi_i,\varpi_i}$ turn out to be the frozen cluster variables of $\mathbb{C}[L^{c,e}]$. Moreover in this special case the cluster structure is trivial and all the cluster variables are frozen. To summarize, we have that
\[
 \mathbb{C}[L^{c,e}] = \mathbb{C}[\Delta_{\varpi_i,\varpi_i}^{\pm 1}\mid 1\le i \le r],
\]
and that the homomorphism of algebras 
\[
\begin{array}{rcl}
   \sigma^*:  \mathbb{C}[L^{c,e}] & \longrightarrow & \mathbb{C}[T]  \\
    \Delta_{\varpi_i,\varpi_i} & \longmapsto & e^{\varpi_i} 
\end{array}
\]
is an isomorphism.
Dually, we obtain an isomophism of varieties $\sigma : T \to L^{c,e}$. We will denote by $\phi := \psi\circ \sigma$
the composition $\phi : T \to A$. The dual map $\phi^*: \mathbb{C}[A] \to \mathbb{C}[T]$ allows to express elements of $\mathbb{C}[A]$, like the fundamental characters $\chi_i$ restricted to $A$, in terms of the characters of~$T$. 

\begin{example}\label{exa4}
{\rm
We continue Example~\ref{exa1}. The double coset $U\overline{c}U$ consists of all elements
of $SL(n)$ of the form
\[
\pmatrix{
 * & *  &\cdots & * & *\cr
 1 & *  & \cdots & * & *\cr
 0 & 1  & \ddots &  & *\cr 
 \vdots&\ddots&\ddots&\ddots & \vdots\cr
 0 & \cdots  & 0 & 1 & *
},
\]
where $*$ denotes an arbitrary element of $\mathbb{C}$.
Its subset $L^{c,e}$ consists of 
all elements of the form
\[
f = 
\pmatrix{
 \beta_1 & 0  &\cdots & \cdots & 0\cr
 1 & \beta_1^{-1}\beta_2  & \ddots &  & \vdots\cr
 0 & 1  & \ddots & \ddots & \vdots\cr 
 \vdots&\ddots&\ddots&\beta_{r-1}^{-1}\beta_r & 0\cr
 0 & \cdots & 0 & 1 & \beta_r^{-1}
},\qquad (\beta_i\in \mathbb{C}^*).
\]
Put 
\[
a:=\psi(f) =
\pmatrix{
 \theta_1 & -\theta_2  &\cdots & (-1)^{r-1}\theta_{r} &(-1)^{r}\cr
 1 & 0  & \cdots & 0 &0\cr
 0 & 1  & \ddots & & \vdots\cr 
 \vdots&\ddots&\ddots&0 & \vdots\cr
 0 & \cdots & 0 & 1& 0
}, \qquad (\theta_i\in \mathbb{C}).
\]
By a slight abuse of notation, let us also denote by $\theta_i$ (\resp $\beta_j$) the 
elements of $\mathbb{C}[A]$ (\resp $\mathbb{C}[L^{c,e}]$) whose evaluation at $a\in A$ (\resp $f\in L^{c,e}$)
is equal to the coordinate $\theta_i$ (\resp $\beta_j$).
Since $a$ and $f$ are conjugate, to express the $\theta_i$'s in terms of the $\beta_j$'s we can use the equalities
$\chi_i(a) = \chi_i(f) \ (1\le i \le r)$. For instance, for $G=SL(3)$ we get :
\[
\psi^*(\theta_1) = \beta_1 + \frac{\beta_2}{\beta_1} + \frac{1}{\beta_2},\qquad
\psi^*(\theta_2) = \beta_2 + \frac{\beta_1}{\beta_2} + \frac{1}{\beta_1},
\]
and for $G=SL(4)$:
\[
\psi^*(\theta_1) = \beta_1 + \frac{\beta_2}{\beta_1} + \frac{\beta_3}{\beta_2} + \frac{1}{\beta_3},\quad
\psi^*(\theta_2) = \beta_2 + \frac{\beta_1\beta_3}{\beta_2} + \frac{\beta_3}{\beta_1} +  \frac{\beta_1}{\beta_3} +\frac{\beta_2}{\beta_1\beta_3}+ \frac{1}{\beta_2},%\quad
\]
\[
\psi^*(\theta_3) = \beta_3 + \frac{\beta_2}{\beta_3} + \frac{\beta_1}{\beta_2} + \frac{1}{\beta_1}.
\]
Let us denote by $y_j := e^{\varpi_j} \in \mathbb{C}[T]$ the characters of $T$ corresponding to the fundamental 
weights~$\varpi_j$. Note that for $f\in L^{c,e}$, we have $\beta_j = \Delta_{\varpi_j,\varpi_j}(f)$.
This means that $\sigma^*(\beta_j) = y_j$. 
So replacing $\beta_j$ by $y_j$ in the above formulas gives an expression for
$\phi^*(\theta_i)$.
Recalling that for $G=SL(n)$ we have $\theta_i = \chi_i$, we recognize the classical expressions of the fundamental characters $\chi_i$ of $SL(n)$ as Laurent polynomials in the variables $y_j$. 
\qed 
}
\end{example}

\subsection{Beyond Steinberg}\label{subsec_1.3}

We still need to explain the precise relation between the functions $\theta_i$ and $\chi_i$ 
for a general $G$.
In \cite[\S 7.4]{S}, Steinberg notes that for $G = SL(n)$ these two functions are equal.
Then he writes : 
\emph{``A similar situation exists in the general case.''} 
More precisely, he shows \cite[\S7.14]{S} that there exist polynomials $f_i$ (\resp $g_i$) 
with integral coefficients in the variables $\theta_j$ (\resp $\chi_j$) with $j\not = i$ such that 
\[
 \chi_i = \theta_i + f_i,\qquad \theta_i = \chi_i + g_i.
\]
This is enough to prove that the map $\chi_{|A}$ of Theorem~\ref{thm1} is an isomorphism. 
But Steinberg does not give an explicit expression for the polynomials $f_i$ or $g_i$.

We will now provide a representation-theoretic interpretation of these relations when $G$ is of Cartan type 
$A$, $D$, $E$. This will lead to an algorithm for calculating the polynomials $g_i$.
In fact we are going to put this problem into a broader context by introducing an infinite 
family of functions generalizing the functions $\theta_i$.

From now on, we thus assume that $G$ is a simple and simply connected group of type $A$, $D$,~$E$.

\begin{Def}
For $i= 1,\ldots , r,$ and $k \ge 1$, we denote by $\theta_{i,k}\in \mathbb{C}[A]$ the function defined by
\[
 \theta_{i,k}(a) := \Delta_{\varpi_i,\varpi_i}\left(a^k\right),\qquad (a\in A).
\]
\end{Def}

The $\theta_{i,k}$'s satisfy the following remarkable system of equations.

\begin{Prop}\label{prop6}
Put $\theta_{i,0} := 1\ (1\le i \le r)$. We then have for $k\ge 1$ and $1\le i \le r$,
 \begin{equation}\label{eq2}
  \theta_{i,k}^2 = \theta_{i,k-1}\theta_{i,k+1} + \prod_{j :\ c_{ij} = -1} \theta_{j,k},
 \end{equation}
where $C = (c_{ij})_{1\le i,j\le r}$ denotes the Cartan matrix of $G$.
\end{Prop}

\begin{proof}
This proposition is a degenerate version of \cite[Proposition 7.1]{FL}. For the convenience of the reader
we give a self-contained proof adapting the ideas of \cite{FL}.

Let $a\in A$. By definition of $A = U(c^{-1})\overline{c} = (U\overline{c})\cap(\overline{c}U^-) $, 
we can write $a = u\overline{c} = \overline{c}u'$ with $u\in U$ and $u'\in U^-$.
Using well-known properties of generalized minors, as described for instance in \cite[\S 2.3]{FZ}, we have for every $i = 1,\ldots,r$,
\[
\Delta_{\varpi_i,\varpi_i}(a^k) = \Delta_{\varpi_i,\varpi_i}(u'a^k) = 
\Delta_{\varpi_i,\varpi_i}(\overline{c}^{-1}\overline{c}u'a^k\overline{c}^{-1}\overline{c}) =
%\Delta_{c(\varpi_i),\varpi_i}(a^{k+1}) =
\Delta_{c(\varpi_i),c(\varpi_i)}(a^{k+1}\overline{c}^{-1}).
\]
Similarly
\[
\Delta_{\varpi_i,\varpi_i}(a^k) = 
\Delta_{\varpi_i,\varpi_i}(a^{k+1}\overline{c}^{-1}u^{-1}) =
\Delta_{\varpi_i,\varpi_i}(a^{k+1}\overline{c}^{-1}),
\]
%and 
\[
\Delta_{\varpi_i,\varpi_i}(a^{k-1}) = 
\Delta_{\varpi_i,\varpi_i}(\overline{c}^{-1}\overline{c}u'a^{k-1}) =
\Delta_{c(\varpi_i),\varpi_i}(a^{k}) 
\]
\[
\qquad \quad=
\Delta_{c(\varpi_i),\varpi_i}(a^{k+1}\overline{c}^{-1}u^{-1}) =
\Delta_{c(\varpi_i),\varpi_i}(a^{k+1}\overline{c}^{-1}),
\]
and 
\[
\Delta_{\varpi_i,\varpi_i}(a^{k+1}) =
\Delta_{\varpi_i,\varpi_i}(a^{k+1}\overline{c}^{-1}\overline{c}) =
\Delta_{\varpi_i,c(\varpi_i)}(a^{k+1}\overline{c}^{-1}).
\]

Let $1\le i,j\le r$ be such that $c_{ij} = -1$. We set $a_{ij} = 1$ if $s_j$ precedes $s_i$ in a reduced decomposition of $c$, and otherwise $a_{ij} = 0$.
Now, the following generalized minor identities follow from \cite[Theorem 1.17]{FZ}:
\[
\Delta_{c(\varpi_i),c(\varpi_i)}\Delta_{\varpi_i,\varpi_i} 
=
\Delta_{c(\varpi_i),\varpi_i}\Delta_{\varpi_i,c(\varpi_i)}
\qquad\qquad\qquad\qquad\qquad\qquad
\]
\begin{equation}\label{eq3n}
\qquad\qquad +\quad
\prod_{j :\ c_{ij} = -1, a_{ij} = 0} \Delta_{\varpi_j,\varpi_j}
\prod_{j :\ c_{ij} = -1, a_{ij} = 1} \Delta_{c(\varpi_j),c(\varpi_j)}.
\end{equation}
Evaluating these identities at the element $a^{k+1}\overline{c}^{-1}$ we get 
the desired relations (\ref{eq2}) evaluated at $a$. 
\end{proof}

It turns out that the functional relations of Proposition~\ref{prop6} are well-known in mathematical physics.
They first appeared in a paper of Kirillov and Reshetikhin \cite{KR}, and were later given the name of $Q$-system.
It was conjectured in \cite{KR} that the characters of certain finite-dimensional $\mathfrak{g}$-modules
(where $\mathfrak{g}$ is the Lie algebra of $G$) are solutions of this $Q$-system.
These $\mathfrak{g}$-modules, known as Kirillov-Reshetikhin modules, arise in the resolution of XXX-type integrable spin chains models by the quantum inverse 
scattering method (see the survey paper \cite{KNS}). In fact they are irreducible modules over the Yangian
$Y(\mathfrak{g})$ of $\mathfrak{g}$, but they can also be regarded as $\mathfrak{g}$-modules via the natural
homomorphism $U(\mathfrak{g}) \to Y(\mathfrak{g})$, and as such they are no longer irreducible in general.
The fact that the characters of the Kirillov-Reshetikhin modules satisfy the 
$Q$-system follows from a result on representations of quantum affine algebras proved much later by Nakajima \cite{N}.

Indeed, the same characters can also be regarded as the characters of the 
$U_q(\mathfrak{g})$-modules obtained by restriction of certain irreducible 
finite-dimensional $U_q(\widehat{\mathfrak{g}})$-mo\-du\-les. Here $U_q(\widehat{\mathfrak{g}})$ denotes the quantum affine algebra associated with $\mathfrak{g}$, and the quantum parameter $q$ is assumed to be a nonzero complex number,
not a root of unity. It is known that the Grothendieck rings of the categories of finite-dimensional representations of $Y(\mathfrak{g})$ and $U_q(\widehat{\mathfrak{g}})$ are isomorphic, and that the simple $Y(\mathfrak{g})$-modules and $U_q(\widehat{\mathfrak{g}})$-modules have the same parametrization and the same $q$-characters.
The notion of $q$-character was introduced independently by Knight for $Y(\mathfrak{g})$, and by Frenkel and 
Reshetikhin for $U_q(\widehat{\mathfrak{g}})$, but they turn out to coincide, see \cite{FR}.
The main features of $q$-characters will be recalled in \S\ref{subsec3.10}.

In the sequel, we will work with $U_q(\widehat{\mathfrak{g}})$-modules rather than with 
$Y(\mathfrak{g})$-modules, as in \cite{FL}. In particular we will regard the Kirillov-Reshetikhin modules as a family of simple $U_q(\widehat{\mathfrak{g}})$-modules, denoted by $W^{(i)}_{k, a}$. 
Here $k \ge 1$, $1\le i \le r$, and 
$a\in \mathbb{C}^*$ is the so-called spectral parameter. 
The Kirillov-Reshetikhin modules corresponding to $k=1$ are the quantum affine analogues of the fundamental
$G$-modules, and are therefore called fundamental $U_q(\widehat{\mathfrak{g}})$-modules. They are also denoted by $L(Y_{i,a})$
when one wants to emphasize their highest weight monomial $Y_{i,a}$. 
The restriction of $W^{(i)}_{k, a}$ to $U_q(\mathfrak{g})$ is
independent of $a$, and its character is denoted by $Q^{(i)}_k$. 

\medskip
We can now state:

\begin{Prop}\label{prop7}
The function $\theta_{i,k}$, regarded as a character of $\mathfrak{g}$, is equal to $Q^{(i)}_k$. 
In particular, the function $\theta_i = \theta_{i,1}$ is equal to the character of the fundamental 
$U_q(\widehat{\mathfrak{g}})$-modules $W^{(i)}_{1, a} = L(Y_{i,a})$ corresponding to the fundamental weight 
$\varpi_i$.
\end{Prop}

The proof of Proposition~\ref{prop7} will be given in \S\ref{subsec3.11} below. 

\begin{example}
{\rm
Using Proposition~\ref{prop7}, we can extract examples of non trivial relations between $\theta_i$ and $\chi_i$
from the rich literature on Kirillov-Reshetikhin modules. The following examples are taken from \cite[Section 6]{CP}.

For $G$ of type $D_4$, we have :
\[
\theta_1 = \chi_1,\quad \theta_2 = \chi_2 + 1,\quad \theta_3 = \chi_3,\quad \theta_4 = \chi_4.
\]
Here, the trivalent vertex of the Dynkin diagram of $G$ is labelled by 2 .
The second equality means that the fundamental module $W^{(2)}_{1, a}= L(Y_{2,a})$ of $U_q(\widehat{\mathfrak{g}})$, regarded
as a $U_q(\mathfrak{g})$-module, decomposes as the corresponding fundamental $U_q(\mathfrak{g})$-module plus
a copy of the trivial representation.

\medskip
For a generalization of these formulas to $G$ of type $D_n$, see \cite[Section 6.2]{CP}.

\medskip
For $G$ of type $E_8$, we have for the ``simplest'' vertices of the Dynkin diagram :
\[
\theta_1 = \chi_1+\chi_8 + 1,\quad \theta_7 = \chi_7 + 2\chi_8 + \chi_1 + 1,\quad \theta_8 = \chi_8 + 1.
\]
\qed}
\end{example}

Frenkel and Mukhin \cite{FM} have described an algorithm to calculate the $q$-characters of the 
Kirillov-Reshetikhin modules. In \cite{HL2} another algorithm was given, based on the cluster algebra structure
of the Grothendieck ring. The specialization $Y_{i,a} \mapsto y_i = e^{\varpi_i}$ in the $q$-character of a module $M$ 
gives the classical character of $M$, so using these algorithms and taking into account Proposition~\ref{prop7}, it is possible to express every function $\theta_{i,k}$ as a Laurent polynomial in the variables $y_i$. 
On the other hand, the fundamental characters $\chi_j$ can be expressed in terms of the $y_i$'s by means of the Weyl character formula. So by comparison one can in principle calculate the polynomials $g_i$ of Steinberg. But even with a computer, this would be a formidable task for $G$ of type $E_8$.
In fact the decomposition of $\theta_{i}$ as a sum of irreducible characters of $G$ is known for all types $E_{6,7,8}$
and all $i$'s, but it was computed by a different method, see \cite[Appendix A]{HKOTY}.

%%%%%%%%%%%%%%%%%%%%%%%

\section{The $q$-difference Drinfeld-Sokolov reduction and Miura transformation}\label{sect2}

%%%%%%%%%%%%%%%%%%%%%%%

In 1998, Frenkel, Reshetikhin, Semenov-Tian-Shansky and Sevostyanov published papers on the $q$-difference Drinfeld-Sokolov
reduction \cite{FRS, SS}, and noted that one of their results could be regarded as an analogue of Steinberg's cross-section theorem in which $G$ is replaced by its formal loop group.

Let $\mathbb{C}((z))$ be the field of Laurent series in one variable $z$, and let $LG$ denote the formal loop group whose $\mathbb{C}$-rational points are the  $\mathbb{C}((z))$-rational points of $G$. 
We will denote by $LB$, $LB^-$, $LU$, $LU^-$ the Borel subgroups of $LG$ and their unipotent subgroups. By analogy with the inclusion $A \subset U\overline{c}\,U$ considered in \S\ref{subsec:1.2}, one can consider the corresponding subspace 
\[
\mathcal{A} := (LU \overline{c}) \cap (\overline{c}LU^-)\subset LU\,\overline{c}\,LU. 
\]
Fix $q\in \mathbb{C}^*$ not a root of unity. The action of $G$ on itself by conjugation is replaced by the action of $LG$ on itself by \emph{$q$-gauge transformation}:
\[
 h(z)\cdot g(z) := h(qz)g(z)h(z)^{-1},\qquad (h(z),\ g(z) \in LG). 
\]
We then have the following analogue of the first part of Theorem~\ref{thm1}. 

Consider the restricted action of $LU$ on $LG$ by $q$-gauge transformation.
The double coset $\mathcal{M} := LU\,\overline{c}\,LU$ is stable under this action.
\begin{Thm}[\cite{FRS,SS}]\label{thm:cross-sec loop}
The action of $LU$ on $\mathcal{M}$ by $q$-gauge transformation is free, and $\mathcal{A}$ is a cross-section.   
\end{Thm}

The corresponding quotient map $\pi : \mathcal{M} \to \mathcal{A}$ 
%mapping $g(z)$ to the unique element$a(z)$ of $A((z))$ on its $U((z))$-orbit 
is called the \emph{$q$-difference Drinfeld-Sokolov reduction}, because it can be seen as an analogue 
of a classical construction of Drinfeld and Sokolov, in which differential operators
are replaced by $q$-difference operators.

As in \S\ref{subsec:1.2}, one can consider the reduced double Bruhat cell
\[
 \mathcal{L}^{c,e} := \left(LU\,\overline{c}\,LU\right) \cap LB^- = \mathcal{M}\cap LB^-.
\]
The restriction of $\pi$ to $\mathcal{L}^{c,e}$ is called the \emph{$q$-difference Miura transformation},
and we will denote it by $\Psi : \mathcal{L}^{c,e} \to \mathcal{A}$.
%Note that, since the action of $U((z))$ on $U((z))\,\overline{c}\,U((z))$ by $q$-gauge transformation is free, the map $\Psi$ is injective.

\begin{example}\label{exa10}
{\rm
Let $G = SL(n)$. We keep the notation of Examples~\ref{exa1} and \ref{exa4}.
The double coset $LU\,\overline{c}\,LU$ consists of all elements
of $LG$ of the form
\[
\pmatrix{
 * & *  &\cdots & * & *\cr
 1 & *  & \cdots & * & *\cr
 0 & 1  & \ddots &  & *\cr 
 \vdots&\ddots&\ddots&\ddots & \vdots\cr
 0 & \cdots  & 0 & 1 & *
}.
\]
%where $*$ denotes an arbitrary element of $\mathbb{C}((z))$.
The cross-section $\mathcal{A}$ consists of all elements of the form 
%\begin{equation}%\label{eq1}
\[
a(z)=
\pmatrix{
 \theta_1(z) & -\theta_2(z)  &\cdots & (-1)^{r-1}\theta_{r}(z) &(-1)^{r}\cr
 1 & 0  & \cdots & 0 &0\cr
 0 & 1  & \ddots & & 0\cr 
 \vdots&\vdots&\ddots&\ddots & \vdots\cr
 0 & 0 & \cdots & 1& 0
},
\ \ (\theta_1(z),\ldots, \theta_{r}(z) \in \mathbb{C}((z))).
\]

These matrices give rise to $G$-valued first order $q$-difference equations of the form
\begin{equation}\label{eq3}
g(qz) = a(z)g(z)\quad  \Longleftrightarrow \quad (D-a(z))g(z) = 0,\qquad (g(z)\in LG), 
\end{equation}
where $D$ is the $q$-difference operator $Dg(z) = g(qz)$.
In other words, the space $\mathcal{A}$ can be identified with a space of $G$-valued first order $q$-difference operators.
Note that $g(z)$ is a solution of (\ref{eq3}) if and only if the $q$-gauge action of $g(z)^{-1}$
on $a(z)$ gives the unit element of $LG$. Note also that this is equivalent to saying that
if $(\varphi_1(z),\ldots,\varphi_n(z))$ is the last row of the matrix $g(z)$, then the formal series 
$\varphi_i(z)$ form a fundamental system of solutions with $q$-Wronskian equal to 1 of the scalar 
$q$-difference equation
of order~$n$:
\[
\varphi(q^nz) = \theta_1(z)\varphi(q^{n-1}z) %- \theta_2(z)\varphi(q^{n-2}z) 
+ \cdots 
+(-1)^{n-2}\theta_{n-1}(z)\varphi(qz) + (-1)^{n-1}\varphi(z). 
\]
Hence, $\mathcal{A}$ can also be identified with a space of scalar $q$-difference operators of order~$n$.
%}
%\end{example}

The reduced double Bruhat cell $\mathcal{L}^{c,e}$ consists of 
all elements of the form
\[
f(z) = 
\pmatrix{
 \beta_1(z) & 0  &\cdots & \cdots & 0\cr
 \cr
 1 & \displaystyle\frac{\beta_2(z)}{\beta_1(z)}  & \ddots &  & \vdots\cr
 0 & 1  & \ddots & \ddots & \vdots\cr 
 \cr
 \vdots&\ddots&\ddots&\displaystyle\frac{\beta_r(z)}{\beta_{r-1}(z)} & 0\cr
 0 & \cdots & 0 & 1 & \displaystyle\frac{1}{\beta_r(z)}
},\qquad (\beta_i(z)\in \mathbb{C}((z))^\times).
\]
Here, for $1\le i \le r = n-1$, we have $\Delta_{\varpi_i,\varpi_i}(f(z)) = \beta_i(z)$.
%, and $\Delta_{w_0(\varpi_i),w_0(\varpi_i)}(f(z)) = \beta_{n-i}(z)^{-1}$.

If $a(z)= \Psi(f(z))$, an elementary calculation allows to express the coordinates $\theta_i(z)$ of $a(z)$
in terms of the coordinates $\beta_i(z)$ of $f(z)$.
For example if $G = SL(2)$ we obtain:
\[
 \theta_1(z) = \beta_1(z) + \frac{1}{\beta_1(qz)},
\]
and for $G = SL(3)$:
\[
\theta_1(z) = \beta_1(z) + \frac{\beta_2(qz)}{\beta_1(qz)} + \frac{1}{\beta_2(q^2z)},
\qquad
\theta_2(z) = \beta_2(z) + \frac{\beta_1(z)}{\beta_2(qz)} + \frac{1}{\beta_1(qz)}.
\]
\qed
}
\end{example}

Thus we see in this example that the $q$-difference Miura transformation $\Psi$ is given by a $q$-difference loop analogue of the morphism $\psi$ of \S\ref{subsec:1.2} expressing the class functions $\theta_i$ in terms of the characters of $T$. Moreover, replacing $\beta_j(q^kz)$ by $Y_{j,q^{2k-j}}$ we get precisely the expressions
of the fundamental $q$-characters of $U_q(\mathfrak{\widehat{g}})$ in type $A$ (compare Example~\ref{ex-29} below). In fact, Frenkel and Reshetikhin note in \cite[\S 8.7]{FR} that, at least conjecturally, the Grothendieck ring of the category of finite-dimensional $U_q(\mathfrak{\widehat{g}})$-modules ``\emph{can be obtained by the $q$-difference Drinfeld-Sokolov reduction. This gives one an
alternative method to find the $q$-characters of irreducible representations''.}
After illustrating this method in the case of $G=SL(n)$, they conclude their paper by predicting that ``\emph{one can probably use the geometry of the orbit space $\mathcal{A} = \mathcal{M}/LU$ to study this question\footnote{The question of calculating the irreducible $q$-characters. In \cite{FR}, the orbit space is denoted by $M^J_{n,q}/LN$.}. 
This method can also be applied to other simply-laced~$\mathfrak{g}$''}.

To summarize, Frenkel and Reshetikhin argue that the $q$-difference Miura transformation  
%one can identify the (complexified) Gro\-then\-dieck ring of finite-dimensional $U_q(\mathfrak{\widehat{g}})$-modules
%with a ring of functions on the cross-section $\mathcal{A}$ in such a way that 
%the natural coordinate functions $\theta_i(z)$ on $\mathcal{A}$ would coincide 
%under a simple change of variables with the $q$-characters of the fundamental $U_q(\widehat{\mathfrak{g}})$-modules. 
%%via the difference Drinfeld-Sokolov reduction and the difference Miura transform. 
could be regarded as a $q$-analogue for loop groups of the second part of Steinberg's theorem,
in which fundamental characters of $G$ get replaced by fundamental $q$-characters of
$U_q(\widehat{\mathfrak{g}})$.

%%%%%%%%%%%%%%%%%%%%%%%%%%%%%%%%%%%
\section{Bands}
%%%%%%%%%%%%%%%%%%%%%%%%%%%%%%%%%%%

\label{sect3}

\subsection{The subcategory $\mathcal{C}_\mathbb{Z}$}

The category $\mathcal{C}$ of finite-dimensional $U_q(\widehat{\mathfrak{g}})$-modules 
has attracted a lot of attention during the past 30 years, see for example \cite{CP,FR,N1,K}.
This category is huge: its simple objects are parametrized by the monoid of dominant loop-weights 
\[
 \widehat{P}_+ := \bigoplus_{1\le i \le r;\ z\in \mathbb{C}^*} \mathbb{N} (\varpi_i, z),
\]
whose parameter set contains discrete parameters $i$ as well as continuous parameters~$z$.
However, as explained in \cite{HL1}, one can consider a monoidal subcategory $\mathcal{C}_\mathbb{Z}\subset \mathcal{C}$
whose simple objects are parametrized by a discrete set, and which already contains 
all the interesting combinatorial information about simple objects of $\mathcal{C}$ and their
tensor products.

This suggests that one could replace the cross-section $\mathcal{A}$ of $LU\,\overline{c}\,LU$
by a discrete analogue of the form $A^\mathbb{Z}$, whose coordinate ring would naturally be identified with the complexified Grothen\-dieck ring of the subcategory $\mathcal{C}_\mathbb{Z}$. This leads us to the definition of bands, which we shall now explain.

\subsection{The scheme $B(G,c)$}

Let $a:=(a(s))_{s\in\mathbb{Z}}$ denote an element of the scheme theoretic product $A^\mathbb{Z}$ of countably many copies of $A$. As in Example~\ref{exa10} above, we can attach to $a$ a
$G$-valued first order linear difference equation
\begin{equation}\label{eq5}
 g(s) = a(s)g(s+1),\qquad (s\in \mathbb{Z}),
\end{equation}
with unknown $b:=(g(s))_{s\in\mathbb{Z}}$ in $G^\mathbb{Z}$. A solution $b$ of (\ref{eq5}) is called a \emph{$(G,c)$-band} for~$a$. 
(Recall that by definition,
the affine space $A$ depends on the choice of a Coxeter element $c$). 
The set of all $(G,c)$-bands $b$ for all possible $a\in A^\mathbb{Z}$ is denoted by $B(G,c)$. In other words:
\begin{Def}\label{def-band}
An element of $B(G,c)$ is a sequence $b=(g(s))_{s\in \mathbb{Z}}$ of elements $g(s)\in G$ such that 
\[
 g(s)g(s+1)^{-1} \in A,\qquad (s\in\mathbb{Z}).
\]
\end{Def}
Clearly, given $a\in A^\mathbb{Z}$ and $g\in G$, there exists a unique $(G,c)$-band $b=(g(s))_{s\in\mathbb{Z}}$ for $a$ satisfying 
the initial condition $g(0) = g$. It is given by 
\begin{eqnarray*}
&& g(0) = g,\quad g(-1) = a(-1)g,\quad g(-2) = a(-2)a(-1)g,\quad \ldots
 \\ 
&& g(1) = a(0)^{-1}g,\quad g(2) = a(1)^{-1}a(0)^{-1}g,\quad \ldots
\end{eqnarray*}
It follows that the map $b \mapsto (g(0), a)$ is a bijection from $B(G,c)$ onto $G\times A^\mathbb{Z}$. 
Now since $G$ is an algebraic variety and $A^\mathbb{Z}$ is an affine space of infinite dimension, we can use this
bijection to prove that: 
\begin{Thm}[\cite{FL}]
The set $B(G,c)$ is endowed with the structure of an infinite-dimen\-sio\-nal affine integral scheme. 
The ring $R(G,c)$ of regular functions on $B(G,c)$ is a unique factorization domain, isomorphic to a
polynomial ring in countably many variables with coefficients in $\mathbb{C}[G]$.
\end{Thm}

\begin{example}\label{exa14}
{\rm
As in Example~\ref{exa1}, let $G=SL(n)$ and let $c=c_{st} = s_1\cdots s_r$ be the standard Coxeter element.

It follows from the description of $A$ given in Equation~(\ref{eq1}) that  
$g(s)g(s+1)^{-1} \in A$ if and only if the first $r$ rows of $g(s+1)$ coincide with
the last $r$ rows of $g(s)$. 
Therefore, we can think of an $(SL(n),c_{st})$-band as an $(\infty\times n)$-array 
\[
\mathrm{B} = \left[b_{ij}\right], \quad (i\in \mathbb{Z},\ 1\le j\le n,\ b_{ij}\in K)
\]
such that every $n\times n$ submatrix
\[
\mathrm{B}(s) = \left[b_{ij}\right], \quad (s\le i\le s+n-1,\ 1\le j\le n)
\]
consisting of $n$ consecutive rows of $\mathrm{B}$ belongs to $SL(n)$. That is, such that 
$
 \det  \mathrm{B}(s) = 1
$
for every $s\in\mathbb{Z}$. Such an array $\mathrm{B}$ corresponds to the $(SL(n),c_{st})$-band $b=(\mathrm{B}(s))_{s\in \mathbb{Z}}$. 
This is the reason for the name \emph{``band''}.

Let 
$\mathbb{C}[X_{ij} \mid  i \in \mathbb{Z}, \ 1\le j \le n]$
be the polynomial ring in the variables $X_{ij}.$ For $s \in \mathbb{Z}$, we set
\[
Y_{s,n} := \det [X_{ij} \mid  s\le i \le s+n-1, \ 1\le j \le n]. 
\]
We can consider the quotient ring
\[
\mathcal{R}_n:=  \mathbb{C}[X_{ij} \mid  i \in \mathbb{Z}, \ 1\le j \le n] / (Y_{s,n}-1 \mid s \in \mathbb{Z}), 
\]
and the associated infinite dimensional affine scheme
\[
\mathcal{B}_n := \mathrm{Spec}(\mathcal{R}_n). 
\]
The previous discussion shows that the set of $(SL(n), c_{st})$-bands can be naturally identified with the set of $\mathbb{C}$-rational points of $\mathcal{B}_n$. \qed
}
\end{example}

\subsection{$G$-action}
\label{subsec-G-act}

The action of $G$ on itself by right translation extends to an action of $G$ on $B(G,c)$:
\[
 (g(s))_{s \in \mathbb{Z}} \cdot h := (g(s)h)_{s\in \mathbb{Z}},\qquad ((g(s))_{s \in \mathbb{Z}} \in B(G,c),\ h\in G).
\]
Indeed $(g(s)h)(g(s+1)h)^{-1} = g(s)g(s+1)^{-1}\in A$.

Under the isomorphism $B(G,c) \simeq G \times A^\mathbb{Z}$, this action reduces to the right action 
of $G$ on the first factor. 
This implies that the morphism from $B(G,c)$ to $A^\mathbb{Z}$ sending $(g(s))_{s \in \mathbb{Z}}$ to 
$(g(s)g(s+1)^{-1})_{s \in \mathbb{Z}}$ factors through an isomorphism from 
the categorical quotient 
\[
B(G,c)\sslash G = \mathrm{Spec}\left(R(G,c)^G\right)
\]
of $B(G,c)$ by this action to $A^\mathbb{Z}$. In other words, we have an isomorphism $R(G,c)^G\simeq \mathbb{C}[A^\mathbb{Z}]$.

The right action of $G$ on $B(G,c)$ induces an interesting left linear action of $G$ on the coordinate ring $R(G,c)$.
We can also restrict the action of $G$ to important subgroups like $U$ or $T$, and consider
the invariant subalgebras $R(G,c)^U$, $R(G,c)^G$ and their weight space decompositions
under the action of $T$.

As mentioned above, $\mathbb{C}[A^\mathbb{Z}]$ can be identified with the complexified Grothendieck ring of the category 
$\mathcal{C}_\mathbb{Z}$. It turns out that $R(G,c)^U$ and $R(G,c)$ are also isomorphic to Grothendieck rings of categories
arising from the representation theory of quantum affine algebras, and this was in fact our initial 
motivation for introducing the scheme $B(G,c)$. 

Note however that a statement like \emph{``the Grothendieck ring of $\mathcal{C}_\mathbb{Z}$ is isomorphic to the coordinate ring of
$A^\mathbb{Z}$''} is not very substantial, since it only amounts to say that both rings are polynomial rings in countably many variables.
To make it meaningful, we need to exhibit a distinguished isomorphism $\mathbb{C}\otimes K_0(\mathcal{C}_\mathbb{Z}) \to \mathbb{C}[A^\mathbb{Z}]$ mapping classes of simple objects of $\mathcal{C}_\mathbb{Z}$ to natural coordinate functions on $A^\mathbb{Z}$. 
To do this we will show that $R(G,c)$ has the additional structure of a cluster algebra, and that $R(G,c)^U$ and $R(G,c)^G$ are cluster subalgebras. Isomorphic cluster structures have already been discovered on the quantum affine algebra side, so by matching them we will obtain the required distinguished isomorphisms.

\section{Cluster structures on $R(G,c)^G$ and $R(G,c)^U$}
\label{sect-invariant-subalg}

\subsection{Coordinate functions}\label{subsect3.4}

To define initial seeds of cluster structures on spaces of bands, we need coordinate systems.
There are two natural families of regular functions on $B(G,c)$, both defined using the generalized minors
introduced in \S\ref{subsec1.1}. 

\begin{Def}
\begin{enumerate}
 \item For $1\le i \le r$, $s\in \mathbb{Z}$, and $v,w\in W$, the function $\Delta^{(s)}_{v(\varpi_i),w(\varpi_i)}$ is
the unique element of $R(G,c)$ such that 
 \[
  \Delta^{(s)}_{v(\varpi_i),w(\varpi_i)}(b) := \Delta_{v(\varpi_i),w(\varpi_i)}(g(s)),
  \qquad
  (b = (g(s))_{s\in \mathbb{Z}} \in B(G,c)).
 \]
 \item
 For $1\le i \le r$, $s\in \mathbb{Z}$, and $k\ge 1$, the function $\theta^{(s)}_{i,k}$ is
 the unique element of $R(G,c)$ such that 
 \[
  \theta^{(s)}_{i,k}(b) := \Delta_{\varpi_i,\varpi_i}(g(s)g(s+k)^{-1}),
  \qquad
  (b = (g(s))_{s\in \mathbb{Z}} \in B(G,c)).
 \]
\end{enumerate}
\end{Def}

For $1\le i \le r$, let $m_i$ denote the smallest integer $k$ such that $c^k(\varpi_i) = w_0(\varpi_i)$,
where $w_0$ is the longest element of $W$.
The defining condition $g(s)g(s+1)^{-1} \in A$ of a band implies the following gluing relations for the 
functions $\Delta^{(s)}_{v(\varpi_i),w(\varpi_i)}$.

\begin{Prop}[\cite{FL}, Proposition 5.11]
    \label{lem: glueing formulas}
    The following formulas hold in the ring $R(G, c)$:
    \begin{equation}
    \label{glue-eq}
 \Delta^{(s)}_{c^k(\varpi_i),\,w(\varpi_i)} = \Delta^{(s+1)}_{c^{k-1}(\varpi_i),\,w(\varpi_i)},
\quad (1\le i\le r,\ 1\le k\le m_i,\ w\in W,\ s\in\mathbb{Z}).
\end{equation}
\end{Prop}

It is easy to see that the functions $\theta^{(s)}_{i,k}$ are $G$-invariant, hence can be regarded as functions on~$A^\mathbb{Z}$.
In particular, for a fixed $s\in\mathbb{Z}$ the functions $\theta^{(s)}_i := \theta^{(s)}_{i,1}\ (1\le i \le r)$ are the polynomial generators
of the coordinate ring of the $s$th copy $A^{(s)}$ of $A$ inside the product $A^\mathbb{Z}$.
Hence we have 
\[
 R(G,c)^G = \mathbb{C}[A^\mathbb{Z}] = \mathbb{C}[\theta^{(s)}_i \mid 1\le i\le r,\ s\in \mathbb{Z}].
\]
The functions $\theta^{(s)}_{i,k}$ are solutions of the following system of functional relations (compare with Proposition~\ref{prop6}).

\begin{Prop}[\cite{FL}, Proposition 7.1]
\label{Prop.7.1}
The functions $\theta^{(s)}_{i,\,k}$ satisfy 
\[
\theta^{(s)}_{i,\,k}\theta^{(s+1)}_{i,\,k} =
\theta^{(s)}_{i,\,k+1}\theta^{(s+1)}_{i,k-1} + \prod_{j:\ c_{ij}=-1} \theta^{(s+a_{ij})}_{j,k},
\qquad
(i\in I,\ s\in \mathbb{Z},\ k\ge 1).
\]
where the integers $a_{ij}$ are as in the proof of Proposition~\ref{prop6}.
\end{Prop}

\begin{proof} The proof follows again from the generalized minor identity (\ref{eq3n}), see \cite{FL}. 
 \end{proof}

\subsection{Cluster structure on $R(G,c)^G$}

Proposition~\ref{Prop.7.1} will allow us to endow $R(G,c)^G$ with the structure of a cluster algebra.

To do this, we first introduce the labelled infinite quiver $\Theta$.
The vertex set of $\Theta$ is $[1,r]\times \mathbb{Z}_{>0}$. 
There is an arrow between two vertices $(i,r)$ and $(j,s)$ if and only if one of the two following
conditions is satisfied:
\begin{itemize}
 \item[(i)] $i = j$ and $|r-s| = 1$, or
 \item[(ii)] $r=s$ and $c_{ij} = -1$.
\end{itemize}
The orientation of these arrows is fixed by the following rules: 
\begin{itemize}
 \item[(iii)] the vertical subquivers $\Theta_i$ with vertex set $\{(i,s) \mid s>0\}$ are in sink-source orientation,
 \item[(iv)] the horizontal subquivers $\Theta^{(s)}$ with vertex set $\{(i,s)\mid i\in I\}$ are in sink-source orientation,
 \item[(v)] if $c_{ij}=-1$ the square supported on vertices $(i,s)$, $(i,s+1)$, $(j,s)$, $(j,s+1)$, is an oriented $4$-cycle:
 \[
\mbox{either}\qquad 
  \begin{array}{ccc}
   (i,s)&\rightarrow &(j,s)\\
   \uparrow&&\downarrow\\
   (i,s+1)&\leftarrow &(j,s+1)
  \end{array}
\qquad\mbox{or}\qquad
  \begin{array}{ccc}
   (i,s)&\leftarrow &(j,s)\\
   \downarrow&&\uparrow\\
   (i,s+1)&\rightarrow &(j,s+1)
  \end{array}
  \]
\end{itemize}

To be precise, there are exactly two quivers satisfying rules (i) to (v) above, and $\Theta$ denotes a fixed choice of one of them.

To each vertex $(i,s)\in \Theta$ we attach the function $\theta_{i,s}^{(n(i,s))}$, where the integer $n(i,s)$
is determined by an explicit rule depending on $c$, see \cite[\S 7.2]{FL}.

\begin{figure}[t]
\[
\xymatrix@-1.0pc{
 &&&\ar[ld]\theta^{(1)}_{4,1}
\\
\theta^{(0)}_{1,1}\ar[dd]& \ar[l]\theta^{(1)}_{2,1}\ar[r] &\theta^{(0)}_{3,1}\ar[dd]
\\
&&&\theta^{(0)}_{4,2}\ar[uu]\ar[dd]&\theta^{(0)}_{5,1}\ar[llu]
\\
\theta^{(0)}_{1,2}\ar[r]& \theta^{(0)}_{2,2}\ar[uu]\ar[dd] &\ar[l]\theta^{(0)}_{3,2}\ar[ur]\ar[rrd]
\\
&&&\ar[ld]\theta^{(0)}_{4,3}&\theta^{(-1)}_{5,2}\ar[uu]\ar[dd]
\\
\theta^{(-1)}_{1,3}\ar[uu]& \ar[l]\theta^{(0)}_{2,3}\ar[r] &\ar[uu]\theta^{(-1)}_{3,3}&\vdots
\\
\vdots&\vdots&\vdots&&\theta^{(-1)}_{5,3}\ar[llu]
\\
&&&&\vdots
}
\]
\caption{\label{Fig9} {\it The first 3 layers of the initial seed $\Theta$ in type $D_5$ for $c=s_2s_4s_1s_3s_5$.}}
\end{figure}

\begin{example}
{\rm
Let $G$ be of type $D_5$ and $c = s_2s_4s_1s_3s_5$.
The corresponding labelled quiver $\Theta$ is displayed in Figure~\ref{Fig9}.
\qed
}
\end{example}

\begin{Thm}[\cite{FL}]\label{Thm.7.3}
The ring $R(G,c)^G$ of $G$-invariant functions on $B(G,c)$ has the structure of a cluster algebra with initial seed given by $\Theta$. 
\end{Thm}

Note that every first step mutation of $\Theta$ is an instance of the functional relations of Proposition~\ref{Prop.7.1}.
This property plays an important role in the proof of Theorem~\ref{Thm.7.3}, which is an application of the Starfish lemma of \cite{FWZ}. Note also that every function $\theta^{(s)}_{i,k}$ is
a cluster variable.

\subsection{$R(G,c)^G$ and the category $\mathcal{C}_\mathbb{Z}$}
\label{sec:5.3}

In order to relate $R(G,c)^G$ with $K_0(\mathcal{C}_\mathbb{Z})$, let us recall some basic results. 
We refer the reader to the survey paper \cite{HLsurvey} for any undefined terminology or notation.
For $1\le i \le r$, let us define integers $\xi_i$ by the following inductive rule:
if there exists $j$ such that $c_{ij} = -1$ and if $s_j$ precedes $s_i$ in a reduced expression of $c$, then 
$\xi_i = \xi_j - 1$, otherwise $\xi_i = 0$. 

It is known that $K_0(\mathcal{C}_\mathbb{Z})$ is the polynomial ring in the classes of the fundamental
$U_q(\widehat{\mathfrak{g}})$-modules: 
\[
L(Y_{i,q^{2s+1-\xi_i}}),\qquad (1\le i \le r,\ s\in \mathbb{Z}).
\]
The fundamental modules belong to the larger family of Kirillov-Reshetikhin modules:
\[
 W^{(i)}_{k,q^{2s+1-\xi_i}} := L(Y_{i,q^{2s+1-\xi_i}}Y_{i,q^{2s+3-\xi_i}}\cdots 
 Y_{i,q^{2s+2k-1-\xi_i}}), 
 \quad (1\le i \le r,\ s\in \mathbb{Z},\ k>0).
\]
It was first conjectured by Kuniba, Nakanishi and Suzuki \cite{KNS1}, and then proved by Nakajima \cite{N}
that the classes $[W^{(i)}_{k,q^p}]\in K_0(\mathcal{C}_\mathbb{Z})$  satisfy the so-called $T$-system of equations:
\[
 [W^{(i)}_{k,q^p}][W^{(i)}_{k,q^{p+2}}] = [W^{(i)}_{k+1,q^p}][W^{(i)}_{k-1,q^{p+2}}]
 \ + \prod_{j:\ c_{ij}=-1} [W^{(j)}_{k,q^{p+1}}].
\]

The following result then follows from Proposition~\ref{Prop.7.1}:

\begin{Prop}\label{Prop.8.1}
Let $\iota : R(G,c)^G \to 
\mathbb{C}\otimes K_0(\mathcal{C}_\mathbb{Z})$ 
be the $\mathbb{C}$-algebra isomorphism defi\-ned by 
\[
 \iota\left(\theta^{(s)}_{i,1}\right) = \left[L(Y_{i,\,q^{2s+1-\xi_i}})\right],\qquad (i\in I,\ s\in \mathbb{Z}). 
\]
Then for every $i\in I$, $s\in \mathbb{Z}$ and $k>0$ we have
\begin{equation}\label{Eq.34}
 \iota\left(\theta^{(s)}_{i,k}\right) = \left[W^{(i)}_{k,\,q^{2s+1-\xi_i}}\right].
\end{equation}
\end{Prop}

The image under $\iota$ of the cluster algebra structure on $R(G,c)^G$ coincides with 
the cluster algebra structure on $K_0(\mathcal{C}_\mathbb{Z})$ introduced in \cite{HL1}.
It was conjectured in \cite{HL1} and proved in \cite{Q} and \cite{KKOP2} that all cluster monomials
are classes of simple objects of $\mathcal{C}_\mathbb{Z}$.

\subsection{Cluster structure on $R(G,c)^U$}\label{subsec3.7}

By definition of the generalized minors, for $1\le i\le r$ and $w\in W$ we have
\[
\Delta_{w(\varpi_i),\varpi_i}(gu) = \Delta_{w(\varpi_i),\varpi_i}(g),\qquad (g\in G,\ u\in U). 
\]
It then follows from the definition of the $G$-action on $B(G,c)$ that 
for every $s\in \mathbb{Z}$, the function $\Delta^{(s)}_{w(\varpi_i),\varpi_i}$ is $U$-invariant in $R(G,c)$.
In particular all functions 
\[
\Delta^{(s)}_{\varpi_i,\varpi_i}, \qquad (1\le i\le r,\ s\in \mathbb{Z})
\]
belong to the subalgebra $R(G,c)^U$.

We now introduce the doubly-infinite labelled quiver $\Xi$. Its vertices are labelled by the functions
$\Delta^{(s)}_{\varpi_i,\varpi_i}\ (1\le i\le r,\ s\in \mathbb{Z})$.
There is an arrow $\Delta^{(s)}_{\varpi_i,\varpi_i} \to \Delta^{(t)}_{\varpi_j,\varpi_j}$ in $\Xi$ if and only if
\[
c_{ij} \not = 0\quad \mbox{and} \quad 2t - \xi_j = 2s - \xi_i + c_{ij}. 
\]

\begin{figure}[t]
\[
\def\objectstyle{\scriptstyle}
\def\lablestyle{\scriptstyle}
\xymatrix@-1.0pc{
&{}\save[]+<0cm,1.5ex>*{\vdots}\restore&{}\save[]+<0cm,1.5ex>*{\vdots}\restore  
&{}\save[]+<0cm,1.5ex>*{\vdots}\restore
\\
&\Delta^{(2)}_{\varpi_1,\varpi_1}\ar[rd]\ar[u]&
&\ar[ld] \Delta^{(2)}_{\varpi_3,\varpi_3} \ar[u]
\\
&&\ar[ld] \Delta^{(1)}_{\varpi_2,\varpi_2} \ar[rd]\ar[uu]&&
\\
&\ar[uu]\Delta^{(1)}_{\varpi_1,\varpi_1}\ar[rd]&
&\ar[ld] \Delta^{(1)}_{\varpi_3,\varpi_3} \ar[uu]
\\
&&\ar[uu]\ar[ld] \Delta^{(0)}_{\varpi_2,\varpi_2} \ar[rd]&&
\\
&\ar[uu]\Delta^{(0)}_{\varpi_1,\varpi_1} \ar[rd] &&\ar[ld] \Delta^{(0)}_{\varpi_3,\varpi_3}\ar[uu]
\\
&&\ar[ld] \ar[uu] \Delta^{(-1)}_{\varpi_2,\varpi_2} \ar[rd]&&
\\
&\ar[uu]\Delta^{(-1)}_{\varpi_1,\varpi_1} \ar[rd] &&\ar[ld] \Delta^{(-1)}_{\varpi_3,\varpi_3}\ar[uu]
\\
&&\ar[ld] \ar[uu]\Delta^{(-2)}_{\varpi_2,\varpi_2} \ar[rd]&&
\\
&\ar[uu]\Delta^{(-2)}_{\varpi_1,\varpi_1} &&\ar[uu] \Delta^{(-2)}_{\varpi_3,\varpi_3} 
\\
&{}\save[]+<0cm,0ex>*{\vdots}\ar[u]\restore&{}\save[]+<0cm,0ex>*{\vdots}\ar[uu]\restore  
&{}\save[]+<0cm,0ex>*{\vdots}\ar[u]\restore
\\
}
\]
\caption{\label{Fig5} {\it The labelled quiver $\Xi$ in type $A_3$ with $c=s_1s_3s_2$.}}
\end{figure}

\begin{example}
{\rm
Let $G$ be of type $A_3$ and $c = s_1s_3s_2$.
The corresponding labelled quiver $\Xi$ is displayed in Figure~\ref{Fig5}. \qed
}
\end{example}

\begin{Thm}[\cite{FL}]%\label{Thm.7.3}
The ring $R(G,c)^U$ of $U$-invariant functions on $B(G,c)$ has the structure of an upper cluster algebra with initial seed given by $\Xi$. 
\end{Thm}

When $G$ is of type $A$, this upper cluster algebra coincides with its genuine underlying cluster algebra. But in type $D$ and $E$ it is not known whether these two cluster algebras coincide or not.

Obviously, $R(G,c)^G$ is a subalgebra of $R(G,c)^U$. Therefore, by the Laurent phenomenon, every element of $R(G,c)^G$ can be written as a Laurent polynomial in the initial cluster variables
$\Delta^{(s)}_{\varpi_i,\varpi_i}$ of $R(G,c)^U$.

\begin{example}
{\rm 
Let $G$ be of type $A_2$ and $c = s_1s_2$. Then
\[
\theta^{(0)}_1 = \frac{\Delta_{\varpi_1,\varpi_1}^{(0)}}{\Delta_{\varpi_1,\varpi_1}^{(1)}} 
+ \frac{\Delta_{\varpi_1,\varpi_1}^{(2)}}{\Delta_{\varpi_1,\varpi_1}^{(1)}}
\frac{\Delta_{\varpi_2,\varpi_2}^{(0)}}{\Delta_{\varpi_2,\varpi_2}^{(1)}}
+ \frac{\Delta_{\varpi_2,\varpi_2}^{(2)}}{\Delta_{\varpi_2,\varpi_2}^{(1)}}.
\]
\qed
}
\end{example}

The cluster expansion with respect to $\Xi$ of a $G$-invariant function has the following
important property.

\begin{Prop}\label{prop26}
Let $\phi\in R(G,c)^G$. The cluster expansion of $\phi$ with respect to the initial seed 
$\Xi$ of $R(G,c)^U$ is a Laurent polynomial in the variables
\[
 \frac{\Delta_{\varpi_i,\varpi_i}^{(s)}}{\Delta_{\varpi_i,\varpi_i}^{(s+1)}},\qquad
 (1\le i \le r,\quad s\in\mathbb{Z}).
\]
\end{Prop}
\begin{proof}
Let $m$ be a Laurent monomial of the cluster expansion of $\phi$, and let $j\in I$.
Recall from \cite[\S6.1]{FL} that $R(G,c)^U$ is $P_+$-graded, where $P_+$ denotes the cone
of dominant weights of $\mathfrak{g}$. In this grading we have 
$\deg(\Delta_{\varpi_i,\varpi_i}^{(s)}) = \varpi_i$. We also have that every element 
of $R(G,c)^G$ is of degree 0. Since the fundamental weights $\varpi_i$ are linearly 
independent, it follows that the number of factors of the form $\Delta_{\varpi_j,\varpi_j}^{(s)}\ (s\in\mathbb{Z})$
in the numerator of $m$ is equal to the number of factors of the same form in the denominator of $m$.
The claim follows immediately. 
\end{proof}

\subsection{Cluster structure on $R(G,c)^{U^-}$}\label{sect3.8}

Clearly, the algebra $R(G,c)^{U^-}$ is isomorphic to $R(G,c)^{U}$. 
Hence it is endowed with an isomorphic cluster algebra structure. 
We only have to replace the $U$-invariant functions
$\Delta^{(s)}_{\varpi_i,\varpi_i}$ of the initial seed $\Xi$ of $R(G,c)^{U}$ by their $U^-$-invariant counterparts $\Delta^{(s)}_{w_0(\varpi_i),w_0(\varpi_i)}$.

It follows that every element of $R(G,c)^G$ can also be written as a Laurent polynomial in 
the initial cluster variables
$\Delta^{(s)}_{w_0(\varpi_i),w_0(\varpi_i)}$ of $R(G,c)^{U^-}$. 

\begin{example}
{\rm 
Let $G$ be of type $A_2$ and $c = s_1s_2$. Then
\[
\theta^{(0)}_1 = \frac{\Delta_{w_0(\varpi_1),w_0(\varpi_1)}^{(-2)}}{\Delta_{w_0(\varpi_1),w_0(\varpi_1)}^{(-1)}} 
+ \frac{\Delta_{w_0(\varpi_1),w_0(\varpi_1)}^{(0)}}{\Delta_{w_0(\varpi_1),w_0(\varpi_1)}^{(-1)}}
\frac{\Delta_{w_0(\varpi_2),w_0(\varpi_2)}^{(-1)}}{\Delta_{w_0(\varpi_2),w_0(\varpi_2)}^{(0)}}
+ \frac{\Delta_{w_0(\varpi_2),w_0(\varpi_2)}^{(1)}}{\Delta_{w_0(\varpi_2),w_0(\varpi_2)}^{(0)}}.
\]
\qed
}
\end{example}

Of course, we have a statement similar to Proposition~\ref{prop26}, in which $R(G,c)^U$ 
is replaced by $R(G,c)^{U^-}$, and
$\displaystyle\frac{\Delta_{\varpi_i,\varpi_i}^{(s)}}{\Delta_{\varpi_i,\varpi_i}^{(s+1)}}$
is replaced by 
$\displaystyle\frac{\Delta_{w_0(\varpi_i),w_0(\varpi_i)}^{(s)}}{\Delta_{w_0(\varpi_i),w_0(\varpi_i)}^{(s+1)}}$.

\subsection{$R(G,c)^U$, $R(G,c)^{U^-}$ and the categories $O^+_\mathbb{Z}$ and $O^-_\mathbb{Z}$}

Let $U_q(\widehat{\mathfrak{b}})$ be the Borel subalgebra of $U_q(\widehat{\mathfrak{g}})$. In \cite{HJ}, Hernandez and Jimbo have introduced 
a category $O$ of representations of $U_q(\widehat{\mathfrak{b}})$ containing all finite-dimensional representations and also many infinite-dimensional ones. Since every finite-dimensional $U_q(\widehat{\mathfrak{g}})$-module remains irreducible by restriction to $U_q(\widehat{\mathfrak{b}})$, the Grothendieck ring of $\mathcal{C}$ can be regarded as a subring of the Grothendieck ring of $O$.

In \cite{HL3}, two subcategories $O^+_\mathbb{Z}$ and $O^-_\mathbb{Z}$ of $O$ were introduced, 
both containing all restrictions of simple objects of $\mathcal{C}_\mathbb{Z}$, 
and it was shown that $K_0(O^+_\mathbb{Z})$ and $K_0(O^-_\mathbb{Z})$ are isomorphic and have the same cluster algebra structure. The building blocks of $O$ are the so-called positive and negative prefundamental representations, together with the one-dimensional representations parametrized by 
$P_\mathbb{Q} := \oplus_{1\le i\le r} \mathbb{Q}\,\varpi_i$. They are denoted respectively by 
\[
L(\Psi_{i,z}),\quad L(\Psi_{i,z}^{-1}),\quad [\varpi],\qquad (1\le i \le r,\ z\in \mathbb{C}^*,\ \varpi\in P_\mathbb{Q}). 
\]
Let $K^+_{0,\mathbb{Z}}$ denote the subring of $K_0(O)$ generated by the classes 
of the positive prefundamental representations 
$L(\Psi_{i,q^{2s-\xi_i}})\ (1\le i\le r,\ s\in\mathbb{Z})$ and $[\varpi]\ (\varpi\in P_\mathbb{Q})$. 
The ring $K_0(O^+_\mathbb{Z})$ can be regarded as a completion of 
$K^+_{0,\mathbb{Z}}$ in which certain infinite sums corresponding to objects of $O^+_\mathbb{Z}$ of infinite
length are allowed, see \cite[\S 5C]{HL3}.

The following proposition follows by comparing the respective initial seeds of the cluster
structures on $R(G,c)^U$ and $K^+_{0,\mathbb{Z}} \subset K_0(O^+_\mathbb{Z})$.

\begin{Prop}\label{prop28}
The assignment 
\[
[(\xi_i/2-s)\varpi_i][L(\Psi_{i,q^{2s-\xi_i}})] \mapsto \Delta^{(s)}_{\varpi_i,\varpi_i}
\]
extends to an injective algebra homomorphism $\mathbb{C}\otimes K^+_{0,\mathbb{Z}} \to R(G,c)^U$ matching
the cluster structures on both sides. 
\end{Prop}

In type $A$ this homomorphism is an isomorphism. In other types it may not be surjective if the 
cluster algebra is strictly contained in its upper cluster algebra.

Similarly, let $K^-_{0,\mathbb{Z}}$ denote the subring of $K_0(O)$ generated by the classes 
of the negative prefundamental representations $L(\Psi_{i,q^{2s-\xi_i}}^{-1})$ and $[\varpi]$.

Using \cite[\S8.2]{FL} and Proposition~\ref{lem: glueing formulas}, we also get the dual statement:
\begin{Prop}\label{prop29}
The assignment 
\[
[(s+m_i-(\xi_i+h)/2)\varpi_{\nu(i)}]\left[L\left(\Psi_{\nu(i),q^{2(s+m_i)-\xi_i-h}}^{-1}\right)\right] 
\mapsto 
\Delta^{(s)}_{w_0(\varpi_i),w_0(\varpi_i)}
\]
extends to an injective algebra homomorphism $\mathbb{C}\otimes K^-_{0,\mathbb{Z}} \to R(G,c)^{U^-}$ matching
the cluster structures on both sides. 
\end{Prop}
Here $\nu$ denotes the involution of $[1,r]$ given by $w_0(\alpha_i) = -\alpha_{\nu(i)}$, and $h$ is the Coxeter number.

\section{$q$-characters and a discrete analogue of the $q$-difference Miura transformation}

\label{sect-discrete-Miura}

\subsection{Cluster expansions and $q$-characters}
\label{subsec3.10}

Recall that Frenkel and Reshetikhin  have associated to every finite-dimensional $U_q(\widehat{\mathfrak{g}})$-module~$M$ its $q$-character $\chi_q(M)$ \cite{FR}.
This is a Laurent polynomial in variables $Y_{i,a} \ (1 \leq i \leq r, \ a \in  \mathbb{C}^*)$,
encoding the dimensions of the loop weight spaces of $M$. The map $M \mapsto \chi_q(M)$
induces an injective homomorphism: 
\[
\chi_q : K_0(\mathcal{C}) \to \mathbb{Z}[Y_{i,a}^{\pm 1}\mid 1\le i \le r,\ a\in \mathbb{C}^*].
\]
We can restrict $\chi_q$ to $\mathcal{C}_\mathbb{Z}$ and get  
an injective homomorphism: 
\[
\chi_q : K_0(\mathcal{C}_\mathbb{Z}) \to \mathbb{Z}\left[Y_{i,q^{2s+1-\xi_i}}^{\pm 1}\mid 1\le i \le r,\ s\in\mathbb{Z}\right].
\]

\begin{example}\label{ex-29}
{\rm 
Let $G$ be of type $A_2$ and $c = s_1s_2$. The $q$-character of the 3-dimensional fundamental module
$L(Y_{1,q})\in \mathcal{C}_\mathbb{Z}$ is given by
\[
\chi_q(L(Y_{1,q})) = Y_{1,q} + Y_{1,q^3}^{-1}Y_{2,q^2} + Y_{2,q^4}^{-1}.
\]
\qed
}
\end{example}

In \cite{HL2,HL3} it was shown that $q$-characters can be obtained as cluster expansions with respect
to a distinguished cluster of the Grothendieck ring. We can reformulate this in terms of $(G,c)$-bands as follows.
In \cite{FH}, Frenkel and Hernandez have described a connection between $q$-characters of objects of $\mathcal{C}$ and the Grothendieck ring of $O^+$. For an object $V$ in $O^+$, let $[V]$ denote its class in $K_0(O^+)$.
They proved the \emph{generalized Baxter's relations}, which state that if $M$
is a finite-dimensional $U_q(\widehat{\mathfrak{g}})$-module, and if we 
perform the following substitution in its $q$-character $\chi_q(M)$: 
\[
Y_{i,a} \to [\varpi_i]\frac{[L(\Psi_{i,q^{-1}a})]}{[L(\Psi_{i,qa})]},
\]
then we get an expression of $[M]$ in the fraction field of $K_0(O^+)$.

\begin{example}
{\rm 
Continuing Example~\ref{ex-29}, if we apply Baxter's relations to $L(Y_{1,q})$
we get
\[
[L(Y_{1,q})] = [\varpi_1]\frac{[L(\Psi_{1,q^{0}})]}{[L(\Psi_{1,q^2})]} + 
[\varpi_2{-}\varpi_1]\frac{[L(\Psi_{1,q^{4}})][L(\Psi_{2,q})]}{[L(\Psi_{1,q^2})][L(\Psi_{2,q^3})]}
+
[-\varpi_2]\frac{[L(\Psi_{2,q^5})]}{[L(\Psi_{2,q^3})]}.
\]
Multiplying this relation by the common denominator, one can interpret it as a 
calculation of the class in $K^+_{0,\mathbb{Z}}$ of the tensor product
$L(Y_{1,q})\otimes L(\Psi_{1,q^2}) \otimes L(\Psi_{2,q^3})$.
\qed}
\end{example}

Let $M\in\mathcal{C}_\mathbb{Z}$. Let $\phi$ (\resp $\phi^-$) denote the image of $[M]$ in $R(G,c)^U$ (\resp  $R(G,c)^{U^-}$) under the homomorphism
of Proposition~\ref{prop28} (\resp Proposition~\ref{prop29}).
In fact, since $M\in\mathcal{C}_\mathbb{Z}$ we know that  $\phi$ (\resp $\phi^-$) belongs to the subalgebra
$R(G,c)^G$.
By Proposition~\ref{prop26}, this shows that the cluster expansion of $\phi$  
is a Laurent polynomial in the variables 
\[
\frac{\Delta_{\varpi_i,\varpi_i}^{(s)}}{\Delta_{\varpi_i,\varpi_i}^{(s+1)}},
\qquad (1\le i \le r,\ s\in\mathbb{Z}).
\]
Comparing this Laurent polynomial with the one obtained from $\chi_q(M)$ 
via the generalized Baxter's relations, we obtain that if we 
perform in $\phi$ the substitution 
\[
\frac{\Delta_{\varpi_i,\varpi_i}^{(s)}}{\Delta_{\varpi_i,\varpi_i}^{(s+1)}} \to 
\frac{[(\xi_i/2-s)\varpi_i][L(\Psi_{i,q^{2s-\xi_i}})]}
{[(\xi_i/2-s-1)\varpi_i][L(\Psi_{i,q^{2s+2-\xi_i}})]}
=
[\varpi_i]\frac{[L(\Psi_{i,q^{2s-\xi_i}})]}
{[L(\Psi_{i,q^{2s+2-\xi_i}})]}
\to Y_{i,q^{2s+1-\xi_i}},
%\quad (1\le i \le r,\ s\in \mathbb{Z}),
\]
then we get the $q$-character $\chi_q(M)$.

There are also generalized Baxter's relations in $K_0(O^-)$, see \cite[\S5.B]{HL3}.
Comparing again Laurent polynomial expansions we obtain that if we 
perform in $\phi^-$ the  substitution 
\begin{equation}\label{Eq9}
\frac{\Delta_{w_0(\varpi_i),w_0(\varpi_i)}^{(s)}}{\Delta_{w_0(\varpi_i),w_0(\varpi_i)}^{(s+1)}} \to 
[-\varpi_{\nu(i)}]\frac{\left[L\left(\Psi_{\nu(i),q^{2(s+m_i)-\xi_i-h}}^{-1}\right)\right]}
{\left[L\left(\Psi_{\nu(i),q^{2(s+m_i+1)-\xi_i-h}}^{-1}\right)\right]}
\to Y_{i,q^{2(s+m_i)+1-\xi_i}},
%\quad (1\le i \le r,\ s\in \Z),
\end{equation}
then we also get the $q$-character $\chi_q(M)$.

\subsection{Discrete Miura transformation and $q$-characters}
\label{subsecMiura}

We can reinterpret the above calculation of $q$-characters in terms of bands by using
a discrete analogue of the $q$-difference Miura transformation $\Psi$ of \S\ref{sect2}. We start with some preparation.

Let $\Omega$ be the subset of $G$  consisting of elements $g$ that admit a  \emph{twisted Birkhoff decomposition}: they can be written as $g = ub$ for some (unique) $u\in U$ and  $b\in B^-.$ 
We recall that $\Omega$ is the principal open subset of $G$ determined by the non-vanishing of the functions $\Delta_{w_0(\varpi_i),w_0(\varpi_i)} \ (1\le i \le r)$.
For a nonnegative integer $n \in \mathbb{N}$, we consider the principal open subset $B(G,c)^{(n)}$ of $B(G,c)$ defined by
\[
B(G,c)^{(n)}:= \{ (g(s))_{s \in \mathbb{Z}} \in B(G,c) \mid g(t) \in \Omega \ \ \mbox{for} \ -n \leq t \leq n \}.
\]
As the $B(G,c)^{(n)}$ are affine schemes, the limit of the inverse system of schemes consisting of the natural inclusions between these open subsets of $B(G,c)$ is represented by an affine scheme, that we denote by $B(G,c)^\circ$. 
One can easily verify that the morphism $B(G,c)^\circ \to B(G,c)$ obtained by composing the natural morphsim $B(G,c)^\circ \to B(G,c)^{(n)}$ with the inclusion $B(G,c)^{(n)} \subseteq B(G,c)$ is independent of $n$ and is a monomorphism. 
Hence, we can identify $B(G,c)^\circ$ with its image under this morphism. This yields that
$$
B(G,c)^\circ= \{(g(s))_{s \in \mathbb{Z}} \in B(G,c) \mid g(s) \in \Omega \ \ \mbox{for} \ s \in \mathbb{Z} \}.
$$
We stress that the morphism $B(G,c)^\circ \to B(G,c)$ is not an open embedding.
Indeed, the only open subset of $B(G,c)$ contained in its image is the empty set. 

Let $R(G,c)^\circ$ be the coordinate ring of $B(G,c)^\circ$. 
We have that $R(G,c)^\circ$ is the localisation of the ring $R(G,c)$ at the multiplicative system consisting of monomials in the elements $\Delta_{w_0(\varpi_i), w_0(\varpi_i)}^{(s)} \ (1\le i \le r,\ s \in \mathbb{Z})$. Moreover, the pullback under the morphism $B(G,c)^\circ \to B(G,c)$ is the localisation map. 
Finally, observe that the open subsets $B(G,c)^{(n)}$  of $B(G,c)$ are stable under the action of $U^-$. 
Therefore, the scheme $B(G,c)^\circ$ inherits a right action of $U^-$, with respect to which the morphism $B(G,c)^\circ \to B(G,c)$ is equivariant.

We now move to the definition of the discrete analogue of the $q$-difference Miura transformation. 
For $(g(s))_{s\in\mathbb{Z}}$ in $B(G,c)^\circ$, we can write the twisted Birkhoff decompositions
\[
 g(s) = u(s)b(s),\qquad (s\in\mathbb{Z}).
\]
Let us also write $a(s):=g(s)g(s+1)^{-1}\in A$. We then obtain that
\begin{equation}
\label{eq:10}
     l(s) := b(s)b(s+1)^{-1} = u(s)^{-1}a(s)u(s+1) \in B^-\cap (U\overline{c}U) = L^{c,e}.
     %,\qquad (s\in\mathbb{Z}).
\end{equation}
Hence, the assignment $(g(s))_{s\in\mathbb{Z}} \mapsto (l(s))_{s \in \mathbb{Z}}$ defines a morphism 
\[
P: B(G,c)^\circ \to (L^{c,e})^\mathbb{Z}.
\]
Notice that we have
\[%begin{equation}\label{Eq6}
\Delta_{w_0(\varpi_i),w_0(\varpi_i)}(l(s)) = 
\frac{\Delta_{w_0(\varpi_i),w_0(\varpi_i)}(b(s))}{\Delta_{w_0(\varpi_i),w_0(\varpi_i)}(b(s+1))}
\]
\begin{equation}\label{Eq6}
\qquad\qquad\qquad\qquad\quad= \frac{\Delta_{w_0(\varpi_i),w_0(\varpi_i)}(g(s))}{\Delta_{w_0(\varpi_i),w_0(\varpi_i)}(g(s+1))}.
\end{equation}
Indeed, the first equality is a consequence of the fact that  $b(s)\in B^-$ for every $s$.
The second one holds because if $g=ub$ is the twisted Birkhoff decomposition of $g \in \Omega$, we have 
\[
 \Delta_{w_0(\varpi_i),w_0(\varpi_i)}(g) = \Delta_{w_0(\varpi_i),w_0(\varpi_i)}(b),
 \qquad (1\le i \le r).
\]
Recall that the functions $\Delta^{(s)}_{w_0(\varpi_i), w_0(\varpi_i)} \, (1\le i \le r, s \in \mathbb{Z})$ form an initial cluster of the cluster structure of $R(G,c)^{U^-}$, and they are therefore algebraically independent. Hence, we deduce from Equation (\ref{Eq6}) that the pullback homomorphism 
\[
P^* : \mathbb{C}[(L^{c,e})^\mathbb{Z}] \to R(G,c)^\circ
\]
is injective.
Thus, for any affine scheme $X$  and any morphism $f: B(G,c)^\circ \to X$ such that the image of $f^*$ is contained in the image of $P^*$, there exists a unique morphism $f': (L^{c,e})^\mathbb{Z} \to X$ such that $f=f' \circ P.$ 

Let $F$ be the restriction to $B(G,c)^\circ $   
of the natural morphism 
\[
B(G,c)  \to B(G,c) \sslash G:= \mathrm{Spec} (R(G,c)^G)
\]
dual to the embedding of algebras $R(G,c)^G \subset R(G,c)$.
We saw in \S\ref{sect3.8} that every element of $R(G,c)^G$ is a Laurent polynomial 
in the variables 
\[
\frac{\Delta_{w_0(\varpi_i),w_0(\varpi_i)}^{(s)}}{\Delta_{w_0(\varpi_i),w_0(\varpi_i)}^{(s+1)}}.
\]
Using again Equation~(\ref{Eq6}), and the fact that $\mathbb{C}[L^{c,e}]$ can be described as the Laurent polynomial ring in the variables
$\Delta_{w_0(\varpi_i),w_0(\varpi_i)}$,
we deduce that the image of the homomorphism $F^*$ is contained in the image of $P^*$. 
Hence, the morphism $F$ factorises uniquely by $P$. In other words, there exists
a unique morphism $F'$ making the following diagram commutative: 
\[
\xymatrix@-1.0pc{
&B(G,c)^\circ \ar[rd]^{P}\ar[rr]^{F}&
&  B(G,c) \sslash G.
\\
&& (L^{c,e})^\mathbb{Z}\ar[ur]^{F'} &&
}
\]
We have seen in \S\ref{subsec-G-act} that the scheme $B(G,c) \sslash G$ can be identified with $A^\mathbb{Z}$ via the isomorphism induced from the morphism $B(G,c) \to A^\mathbb{Z}$ sending a band $(g(s))_{s \in \mathbb{Z}}$ to $(a(s))_{s \in \mathbb{Z}}$, where $a(s)=g(s)g(s+1)^{-1}$. 
We denote by $H : (L^{c,e})^\mathbb{Z} \to A^\mathbb{Z}$ the composition of the morphism $F'$ with this isomorphism $B(G,c) \sslash G \simeq A^\mathbb{Z}$.
We regard the map $H : (L^{c,e})^\mathbb{Z} \to A^\mathbb{Z}$ as \emph{a discrete analogue of the $q$-difference Miura transformation $\Psi$} of \S\ref{sect2}. 
We will further justify this terminology in \S\ref{sec: discrete Gauge}. 

%Let us stress that (by construction) the pullback under $H$ is the $q$-character homomorphism.
Consider now the isomorphism 
\[
 \kappa\colon \mathbb{C}\left[Y_{i,q^{2s+1-\xi_i}}^{\pm 1}\mid 1\le i \le r,\ s\in\mathbb{Z}\right] \to \mathbb{C}\left[(L^{c,e})^\mathbb{Z}\right]
\]
assigning to the variable $Y_{i,q^{2(s +m_i)+1-\xi_i}}$ the function on $(L^{c,e})^\mathbb{Z}$ sending $l \in (L^{c,e})^\mathbb{Z}$ to $\Delta_{w_0(\varpi_i), w_0(\varpi_i)}(l(s))$.
By Equation~(\ref{Eq9}), Equation~(\ref{Eq6}), and by construction of $H$, this fits into the following commutative diagram
\begin{equation}
    \label{eq: H and w character}
\begin{tikzcd}
	{\mathbb{C}[A^\mathbb{Z}] \simeq R(G,c)^G} && {\mathbb{C}[(L^{c,e})^\mathbb{Z}]} \\
	{\mathbb{C} \otimes K_0(\mathcal{C}_\mathbb{Z})} && {\mathbb{C}\left[Y_{i,q^{2s+1-\xi_i}}^{\pm 1}\mid 1\le i \le r,\ s\in\mathbb{Z}\right]}
	\arrow["{H^*}", from=1-1, to=1-3]
	\arrow["\iota", from=1-1, to=2-1]
	\arrow["\kappa",from=2-3, to=1-3]
	\arrow["{\chi_q}", from=2-1, to=2-3]
\end{tikzcd}
\end{equation}

Thus, we have proved:
\begin{Thm}\label{Thm-FR}
Under the above isomorphisms 
\[
{\mathbb{C}[A^\mathbb{Z}] \simeq \mathbb{C} \otimes K_0(\mathcal{C}_\mathbb{Z})} \mbox{ and \ } 
\mathbb{C}\left[(L^{c,e})^\mathbb{Z}\right] \simeq  \mathbb{C}\left[Y_{i,q^{2s+1-\xi_i}}^{\pm 1}\right] 
\]
the pullback $H^*$ of the discrete analogue $H$ of the $q$-difference Miura transformation coincides 
with the $q$-character homomorphism $\chi_q$. \qed
\end{Thm}

This verifies the expectation of Frenkel and Reshetikhin \cite[\S 8]{FR} for all types $A, D, E$
and all Coxeter elements $c$. 
In classical types $A, D$, explicit calculations of the $q$-difference Miura transformation were given in \cite[\S 11]{FR0}
(for an implicit choice of a particular Coxeter element $c$).

\begin{example}
{\rm
Let $G=SL(3)$ and $c = s_1s_2$. The twisted Birkhoff decomposition is computed as follows. Let $g \in \Omega$. In other words,  
\[
 g =
 \pmatrix{
 a&b&c\cr
 d&e&f\cr
 h&i&k
 } 
\]
is such that $k\not = 0$ and $ek-fi\not = 0$. Then there holds
\[
 g = \left(
 \begin{array}{ccc}
 1&{\displaystyle\frac{bk-ci}{ek-fi}}&{\displaystyle\frac{c}{k}}\\[3mm]
 0&1&{\displaystyle\frac{f}{k}}\\[3mm]
 0&0&1 
 \end{array}
 \right)
\left(
 \begin{array}{ccc}
 {\displaystyle\frac{1}{ek-fi}}&0&0\\[3mm]
 {\displaystyle\frac{dk-fh}{k}}&{\displaystyle\frac{ek-fi}{k}}&0\\[3mm]
 h&i&k 
 \end{array}
 \right).
\]
Consider $(g(s))_{s\in\mathbb{Z}}\in B(SL(3),c)$. 
As explained in Example~\ref{exa14}, we can write
\[
g(s) =
\pmatrix{
 b_{s,1}&b_{s,2}&b_{s,3}\cr
 b_{s+1,1}&b_{s+1,2}&b_{s+1,3}\cr
 b_{s+2,1}&b_{s+2,2}&b_{s+2,3}
 },\qquad
 (s\in\mathbb{Z}),
\]
and
\[
a(s) = g(s)g(s+1)^{-1} =
\pmatrix{
 \theta^{(s)}_1&-\theta^{(s)}_2&1\cr
 1&0&0\cr
 0&1&0
},\qquad
 (s\in\mathbb{Z}).
\]
Assume that $(g(s))_{s\in\mathbb{Z}}\in B(SL(3),c)^\circ$, that is,
\[
b_{s,3}\not = 0,\qquad b_{s,2}b_{s+1,3}-b_{s+1,2}b_{s,3}\not = 0,\qquad (s\in \mathbb{Z}). 
\]
Then $g(s)$ has a twisted Birkhoff decomposition $g(s)= u(s)b(s)$, and we can
calculate
\[
b(s)b(s+1)^{-1} =
\]
\[
\left(
 \begin{array}{ccc}
 {\displaystyle\frac{b_{s+2,2}b_{s+3,3}-b_{s+3,2}b_{s+2,3}}{b_{s+1,2}b_{s+2,3}-b_{s+2,2}b_{s+1,3}}}&0&0\\[3mm]
 1&{\displaystyle\frac{(b_{s+1,2}b_{s+2,3}-b_{s+2,2}b_{s+1,3})}{(b_{s+2,2}b_{s+3,3}-b_{s+3,2}b_{s+2,3})}
 \frac{b_{s+3,3}}{b_{s+2,3}}}&0\\[3mm]
 0&1& {\displaystyle\frac{b_{s+2,3}}{b_{s+3,3}}}
 \end{array}
 \right),
\]
that is, writing for short $\Delta^{(s)}_{w_0(\varpi_i)} := \Delta_{w_0(\varpi_i),w_0(\varpi_i)}(g(s))$,
\[
l(s) := b(s)b(s+1)^{-1} =
\left(
\begin{array}{ccc}
 {\displaystyle\frac{\Delta^{(s+1)}_{w_0(\varpi_2)}}{\Delta^{(s)}_{w_0(\varpi_2)}}}&0&0\\[3mm]
 1&{\displaystyle\frac{\Delta^{(s)}_{w_0(\varpi_2)}}{\Delta^{(s+1)}_{w_0(\varpi_2)}}
 \frac{\Delta^{(s+1)}_{w_0(\varpi_1)}}{\Delta^{(s)}_{w_0(\varpi_1)}}}&0\\[3mm]
 0&1& {\displaystyle\frac{\Delta^{(s)}_{w_0(\varpi_1)}}{\Delta^{(s+1)}_{w_0(\varpi_1)}}}
 \end{array}
\right) 
\in L^{c,e}.
\]
Put 
\[
l(s):=
\pmatrix{
l_1^{(s)}&0&0\cr
1&l_2^{(s)}&0\cr
0&1&l_3^{(s)}
}
,\qquad (s\in\mathbb{Z}). 
\]
The discrete Miura transformation $H$ maps $l(s)$ to $a(s)$. One can calculate explicitly
the coordinates $\theta^{(s)}_i$ of $a(s)$ in terms of the coordinates $l^{(s)}_i$ of $l(s)$.
For example one has
\[
\theta^{(s)}_1 = \frac{\Delta_{w_0(\varpi_2)}^{(s+1)}}{\Delta_{w_0(\varpi_2)}^{(s)}}
+ \frac{\Delta_{w_0(\varpi_1)}^{(s)}}{\Delta_{w_0(\varpi_1)}^{(s-1)}}
\frac{\Delta_{w_0(\varpi_2)}^{(s-1)}}{\Delta_{w_0(\varpi_2)}^{(s)}}
+ \frac{\Delta_{w_0(\varpi_1)}^{(s-2)}}{\Delta_{w_0(\varpi_1)}^{(s-1)}} 
=
l^{(s)}_1+l^{(s-1)}_2+l^{(s-2)}_3,\quad (s\in\mathbb{Z}).
\]
\qed
}
\end{example}

\begin{remark}
{\rm
Instead of working with $U^-$-invariants and the twisted Birkhoff decomposition, we could
have chosen to use $U$-invariants and the usual Birkhoff decomposition:
\[
 g = u^-b^+,\qquad (u^-\in U^-,\ b^+\in B)\, 
\]
which is well-defined under the condition $\Delta_{\varpi_i,\varpi_i}(g)\not = 0$ for every
$1\le i \le r$. Then if $(g(s))_{s\in\mathbb{Z}}$ is a $(G,c)$-band such that every $g(s)$ has a Birkhoff 
decomposition, we can write
\[
 g(s) = u^-(s)b^+(s),\qquad (s\in\mathbb{Z},\ u^-(s)\in U^-,\ b^+(s)\in B),
\]
and
\[
 \ell(s) := b^+(s)b^+(s+1)^{-1} = u^-(s)^{-1}a(s)u^-(s+1) \in B^+ \cap (U^-\overline{c}U^-) =: L^{e,c}.
\]
Then we would get a similar diagram
\[
\xymatrix@-1.0pc{
&B(G,c)^\dag \ar[rd]^{}\ar[rr]^{}&
&  B(G,c) \sslash G \simeq A^\mathbb{Z}
\\
&& (L^{e,c})^\mathbb{Z}\ar[ur]^{H'} &&
}
\]
which defines another discrete Miura transformation $H'$ mapping $(\ell(s))_{s\in\mathbb{Z}}\in (L^{e,c})^\mathbb{Z}$ to $(a(s))_{s\in\mathbb{Z}}\in A^\mathbb{Z}$. This choice would give slightly simpler formulas involving the $U$-invariant functions $\Delta^{(s)}_{\varpi_i,\varpi_i}$ instead of the $U^-$-invariant functions $\Delta^{(s)}_{w_0(\varpi_i),w_0(\varpi_i)}$, but on the other hand the choice of $H$ allows a direct comparison with the $q$-difference Miura transformation $\Psi$ of \cite{FR,FRS,SS} (see \S\ref{sect2}). 
}
\end{remark}

\subsection{Discrete gauge action and a cross-section theorem}
\label{sec: discrete Gauge}

Let us consider a discrete analogue of the action by $q$-gauge transformation of \S\ref{sect2}. 
Namely, we let the algebraic group scheme $G^\mathbb{Z}$ act on the left on itself by the \emph{discrete gauge transformation}:
\begin{equation}
    \label{eq:disc Gauge}
g \cdot h := \bigl(g(s)h(s)g(s+1)^{-1}\bigr)_{s \in \mathbb{Z}}, \quad \bigl(g= (g(s))_{s \in \mathbb{Z}}, \ h=(h(s))_{s \in \mathbb{Z}} \in G^\mathbb{Z} \bigr).
\end{equation}
This action is intimately related with $(G,c)$-bands. 
For instance, let $\alpha_e : G^\mathbb{Z} \to G^\mathbb{Z}$ be the morphism corresponding to the orbit of the identity $e \in G^\mathbb{Z}$. 
   In other words, $\alpha_e$ sends a point $(g(s))_{s \in \mathbb{Z}}$ to $(g(s)g(s+1)^{-1})_{s \in \mathbb{Z}}$.
   Then the scheme $B(G,c)$ is by definition the schematic pre-image $\alpha_e^{-1}(A^\mathbb{Z})$ of $A^\mathbb{Z}$.

   Moreover, let $b:=(g(s))_{s \in \mathbb{Z}} \in B(G,c)^{\circ}$, and let $l:=P(b) \in (L^{c,e})^\mathbb{Z}$ and $a:= H(l)$. Then Equation (\ref{eq:10}) implies that $a$ and $l$ are conjugated under the discrete gauge transformation by an element of $U^{\mathbb{Z}}$. 
This motivates us to study this action further.

Put $M:=U \bar{c}U \subseteq G$.
The closed subscheme $M^\mathbb{Z}$ of $G^\mathbb{Z}$ is stable under the restricted action by discrete gauge transformation of $U^\mathbb{Z}$. 
In this section, we will prove the following analogue of the cross-section theorem for the $q$-gauge action (Theorem \ref{thm:cross-sec loop}).

\begin{Thm}
    \label{thm:cross-sec discrete}
    For any $\mathbb{C}$-algebra $R$, the group $U^\mathbb{Z}(R)$ acts freely on $M^\mathbb{Z}(R)$ 
    by discrete gauge transformation, and $A^\mathbb{Z}(R)$ is a cross-section.
\end{Thm}

Theorem~\ref{thm:cross-sec discrete} can be rephrased as follows.

\begin{Cor}\label{cor-34}
The action of $U^\mathbb{Z}$ on $M^\mathbb{Z}$ by discrete gauge transformation induces an isomorphism
of schemes $U^\mathbb{Z} \times A^\mathbb{Z} \to M^\mathbb{Z}$.
\end{Cor}

\begin{proof}
Theorem~\ref{thm:cross-sec discrete} is equivalent to the fact that, for any $\mathbb{C}$-algebra $R$, the 
map $U^\mathbb{Z}(R) \times A^\mathbb{Z}(R) \to M^\mathbb{Z}(R)$ sending $(u,a)$ to $u\cdot a$ is a bijection.
\end{proof}

The proof of Theorem~\ref{thm:cross-sec discrete} requires some preparation. 
We start by recalling a well known result on algebraic groups.
Let us consider an algebraic subgroup $H$ of $U$ which is stable under the conjugation action of $T$.
Then the group $H$ is uniquely determined by the set $\Phi(H)$ of weights of the Lie algebra of $H$.
The following result is well known, and can be deduced for instance from \cite[Proposition 28.1]{Hum}.

\begin{Lem}
    \label{lem:normal forms Hum} The following statements hold.
    \begin{enumerate}
        \item Let $H_1,H_2$ and $H_3$ be $T$-stable subgroups of $U$ such that $\Phi(H_1) \cup \Phi(H_2)= \Phi(H_3)$ and $\Phi(H_1) \cap \Phi(H_2)= \emptyset.$ 
        Then the product of $G$ induces an isomorphism of algebraic varieties $H_1 \times H_2 \to H_3$.
    \item The morphism $U(c^{-1}) \times U \to M$     defined by $(b,d) \longmapsto b \bar{c}d \ (b \in U(c^{-1}), \  d \in U)$ is an isomorphism.
    \end{enumerate}
\end{Lem}

\begin{Lem}
\label{lem:groups Hs} For every $s \in \mathbb{Z}_{\geq 0}$, the following hold.
\begin{enumerate}
    \item  There exists a $T$-stable algebraic subgroup $H_{s,c}$ of $U$ such that 
    \[
    \Phi(H_{s,c})= \bigl( \Phi^+ \cap c\Phi^-) \cup c\biggl( \bigcap_{0 \leq k \leq s+1} c^{-k} \Phi^+ \biggr).
    \]
     \item The morphisms given by multiplication 
    \[
    U(c^{-1}) \times c \biggl( \bigcap_{0 \leq k \leq s+1} c^{-k}U c^k \biggr) c^{-1} \to H_{s,c} \leftarrow c \biggl( \bigcap_{0 \leq k \leq s+1} c^{-k}U c^k \biggr) c^{-1} \times U(c^{-1})
    \]
    are isomorphisms.
    \item The image of the product map
    \[
  \biggl( \bigcap_{0 \leq k \leq s+1} c^{-k}Uc^k \biggr) \times  U(c^{-1}) \to U
    \]
    is contained in $H_{s+1,c}$.
\end{enumerate}
\end{Lem}

\begin{proof}

To prove (1), it is convenient to use the combinatorics of Auslander-Reiten quivers 
as we did in \cite[\S 5.2]{FL}. We will use the same notation as in \cite{FL}. 
In particular $Q$ denotes the Dynkin quiver associated with $c$, and $\mathcal{G}_Q$ its Auslander-Reiten quiver. The vertex set of $\mathcal{G}_Q$ consists of all
isoclasses of indecomposable representations of~$Q$. It is in bijection with $\Phi^+$ by taking dimension vectors. Let $\tau$ be the Auslander-Reiten translation. If $x \in \mathrm{Rep}(Q)$ has dimension
vector $\beta$ then $\tau(x)$ has dimension vector~$c(\beta)$.
Recall that the $i$-th row of $\mathcal{G}_Q$ (the one containing the indecomposable injective representation $I_i$) has length $m_i$.

Following \cite[\S 2.2]{GLS} we say that a multiplicity free direct sum $M$ of indecomposable representations of $Q$ containing 
all the $I_i$'s is a \emph{terminal module} if the subset of vertices of the graph $\mathcal{G}_Q$ corresponding to its indecomposable direct summands is closed under successor. By \cite[\S 3.7]{GLS}, if $M$ is a terminal representation of $Q$,
then there exists a unique $w\in W$ such that the set of dimension vectors of the irreducible
summands of $M$ is equal to 
\[
 \Phi_w := \{\alpha \in \Phi^+ \mid w(\alpha) \in \Phi^-\} = \Phi(U \cap w^{-1} U^- w).
\]
Put 
\[
    \Phi_{s,c}:= ( \Phi^+ \cap c\Phi^-) \cup c\biggl( \bigcap_{0 \leq k \leq s+1} c^{-k} \Phi^+ \biggr),\qquad (s\ge 0).
\]
Then $\Phi_{s,c}$ contains the subset $\Phi^+ \cap c\Phi^-$ of dimension vectors of the indecomposable
injective representations $I_i$ of $Q$. We have 
\[
\Phi_{0,c} = ( \Phi^+ \cap c\Phi^-) \cup (c\Phi^+ \cap \Phi^+) = \Phi^+.
\]
Then 
\[
\Phi_{1,c} = ( \Phi^+ \cap c\Phi^-) \cup (c\Phi^+ \cap \Phi^+ \cap c^{-1}\Phi^+)
\]
is the subset of $\Phi^+$ obtained by removing the leftmost vertex of $\mathcal{G}_Q$ on 
every row $i$ such that $m_i > 1$. Similarly, $\Phi_{s,c}$ is the subset of $\Phi^+$
obtained by removing on every row $i$ of $\mathcal{G}_Q$ the $\min(s,m_i-1)$ leftmost vertices.
Hence if $s \ge \max(m_i\mid 1\le i \le r)$, then we have $\Phi_{s,c} = \Phi^+ \cap c\Phi^-$.
Finally it is clear that for every $s$ the direct sum $M_s$ of indecomposable representations
with dimension vectors in $\Phi_{s,c}$ is a terminal module. This follows 
from the well-known property 
\[
\mathrm{Hom}(\tau^{-k}(P_i), \tau^{-l}(P_j)) = 0, \qquad (k>l\ge0,\ 1\le i,j \le r),
\]
where the $P_i$'s are the indecomposable projective representations.
Indeed, this implies that every arrow of $\mathcal{G}_Q$ starting in a vertex corresponding to an element of $\Phi_{s,c}$
ends in a vertex corresponding to another element of $\Phi_{s,c}$. Hence $\Phi_{s,c}$ is closed under successor. Therefore, there exists $w_{s,c}\in W$ such that
\[
 \Phi_{s,c} = \Phi_{w_{s,c}} = \Phi(H_{s,c}),
\]
where $H_{s,c} = U \cap w_{s,c}^{-1} U^- w_{s,c}$. This proves (1).

Statement (2) is an immediate consequence of Lemma \ref{lem:normal forms Hum}\,(1).
Finally, let $d$ and $b$ be respectively elements of the groups 
$\bigcap_{0 \leq k \leq s+1} c^{-k}Uc^k$ and $U(c^{-1})$.
Lemma \ref{lem:normal forms Hum}\,(1) implies that $U\simeq (U\cap cUc^{-1}) \times (U\cap cU^-c^{-1})$,
so 
\[
\bigcap_{0 \leq k \leq s+1} c^{-k}Uc^k \simeq 
\left(cUc^{-1} \cap \bigcap_{0 \leq k \leq s+1} c^{-k}Uc^k\right)
\times \left(cU^-c^{-1} \cap \bigcap_{0 \leq k \leq s+1} c^{-k}Uc^k\right).
\]
Hence there exist unique elements
\[
d_1 \in cUc^{-1} \cap \bigcap_{0 \leq k \leq s+1} c^{-k}Uc^k, \qquad d_2 \in cU^-c^{-1} \cap \bigcap_{0 \leq k \leq s+1} c^{-k}Uc^k
\]
such that $d=d_1d_2$. 
Then we have that $db=d_1(d_2b)$ and $d_2b \in U(c^{-1})$. 
Hence, statement (3) follows from statement (2).
\end{proof}

The proof of Theorem \ref{thm:cross-sec discrete} requires a limit argument, for which we need to introduce some truncated analogues of the discrete gauge action. 
In particular, for any integer $n \in \mathbb{Z}$, we let $U^{\mathbb{Z}_{\geq n}}$ act on $M^{\mathbb{Z}_{\geq n}}$ by means of the formulas of Equation (\ref{eq:disc Gauge}). 

\begin{Prop}
\label{prop:truncated Gauge}
   Let $R$ be a $\mathbb{C}$-algebra. Fix $m \in M^{\mathbb{Z}_{\geq 0}}(R)$ and $u_\circ \in U(R).$ Then there exists a unique $u \in U^{\mathbb{Z}_{\geq 0}}(R)$ such that $u \cdot m$ belongs to $A^{\mathbb{Z}_{\geq 0}}(R)$ and $u(0)=u_\circ$.
Moreover, the component $u(s)$ of $u$ does not depend on $u_\circ$ whenever $s \geq  \max \{ m_i \mid 1 \leq i \leq r\}.$
\end{Prop}

\begin{proof}
As the $\mathbb{C}$-algebra $R$ is fixed, we will abuse notation and simply write $x \in X$ to denote that an element $x$ belongs to $X(R)$ for some complex scheme $X$.
%\marginpar{\color{red} there is a similar abuse through the text, we often write $b \in B(G,c)$ to denote that $b \in B(G,c)(R)$ for some $\C$-algebra $R$.}
Because of Lemma \ref{lem:normal forms Hum}\,(2), we can write 
\[
m(0)= b(0) \bar{c}x(0)
\]
for some unique $b(0) \in U(c^{-1})$ and $x(0) \in U$. Then we can iteratively write
\[
x(s-1) m(s)=b(s) \bar{c}x(s) \qquad (1\leq s),
\]
for some unique $b(s) \in U(c^{-1})$ and $x(s) \in U$.
We claim that for every $n \in \mathbb{Z}_{\geq 1 }$ there exist elements $d(s) \in U \ (0 \leq s \leq n)$ satisfying the following properties:
\begin{enumerate}
    \item  Let $(u(t))_{0 \leq t \leq n} \in U^{n+1}$ be a sequence verifying $u(0)=u_\circ$. Then the points  $u(s)m(s)u(s+1)^{-1}$ belong to $A$ for $0 \leq s \leq n-1$ if and only if 
    \[
    u(s+1)= d(s) x(s) \qquad(0 \leq s \leq n-1).
    \]
    \item The element $d(s)$ belongs to the group
    \[
\bigcap_{0 \leq k \leq s+1} c^{-k} U c^k.
    \]
\end{enumerate}
The previous claim immediately implies the proposition. 
Indeed, it is clear that for every $n \in \mathbb{Z}_{\geq 1}$ the sequence $(d(s))_{0 \leq s \leq n}$ of the claim is unique. 
Hence, by letting $n$ tend to infinity, we get a well defined sequence $(d(s))_{s \in \mathbb{Z}_{\geq 0}} \in U^{\mathbb{Z}_{\geq 0}}$.
Let $u \in U^{\mathbb{Z}_{\geq 0}}$ such that $u(0)=u_\circ$. 
Then we have that $u \cdot m \in A^\mathbb{Z}$ if and only if $u(s+1)=d(s)x(s)$ for every $s \geq 0$. 
Moreover, the second part of the claim implies that the element $d(s)$ is the identity whenever $s \geq \max \{m_i \mid 1\le i \le r\}$. 
As the elements $x(t) \ (t \ge 0) $ only depend on $m$, the second statement of the proposition follows.

Let us now prove the claim by induction on $n$.
Assume that $n=1$. Let $(u(0), u(1)) \in U^2$ such that $u(0)=u_\circ.$ 
Lemma \ref{lem:normal forms Hum}\,(1) implies that $U\simeq U(c^{-1}) \times (U\cap cUc^{-1})$,
so there exist unique elements $\widetilde{b(0)} \in U(c^{-1})$ and $\widetilde{d(0)} \in U \cap cUc^{-1}$
such that
%By Lemma \ref{lem:groups Hs}, we can write
\[
u(0)b(0)=\widetilde{b(0)} \, \widetilde{d(0)}.
\]
%for some unique $\widetilde{b(0)} \in U(c^{-1})$ and $\widetilde{d(0)} \in U \cap cUc^{-1}$.
Let us set
\[
d(0):= \overline{c}^{-1}  \widetilde{d(0)} \overline c \in U \cap c^{-1}Uc.
\]
Then we have that 
\[
\begin{array}{rl}
     u(0)m(0)= & u(0) b(0) \bar{c}x(0) \\[0.3em]
     = & \widetilde{b(0)} \overline{c} d(0) x(0).
\end{array}
\]
Then Lemma \ref{lem:normal forms Hum}\,(2) implies that the equation
\[
u(0)m(0)=a(0)u(1)
\]
admits a solution $a(0) \in A$ if and only if 
\[
u(1)= d(0)x(0).
\]
Hence, we have proved the claim for $n=0$.
%{
%\color{red} Moreover, we have also showed that with this choice of  $u(1)$, then  $u(0)m(0)u(1)^{-1}%=\widetilde{b(0)} \overline{c}.$
%}

Assume now that $n\geq 1$ and that the sequence $(u(t))_{0 \leq t \leq n+1}$ satisfies $u(0)=u_\circ$ and $u(s)m(s)u(s+1)^{-1} \in A$ for $0 \leq s \leq n-1.$
By inductive hypothesis, there exist elements $d(s) \, (1 \leq s \leq n-1)$ such that 
\[
u(s+1)= d(s) x(s), \qquad d(s) \in \bigcap_{0 \leq k \leq s+1}c^{-k}Uc^k.
\]
By Lemma \ref{lem:groups Hs}, we can write 
\[
d(n-1)b(n)= \widetilde{b(n)} \cdot \widetilde{d(n)}
\]
for some unique 
\[
\widetilde{b(n)} \in U(c^{-1}), \qquad \widetilde{d(n)} \in c \biggl(\bigcap_{0 \leq k \leq n+1} c^{-k}Uc^k \biggr)c^{-1}.
\]
Let us set
\[
d(n):= \overline{c}^{-1}\widetilde{d(n)} \overline{c} \in \bigcap_{0 \leq k \leq n+1} c^{-k}Uc^k.
\]
Then we have that 
\[
\begin{array}{rl}
     u(n)m(n)= & d(n-1) b(n) \bar{c}x(n) \\[0.3em]
     = & \widetilde{b(n)} \overline{c} d(n) x(n).
\end{array}
\]
Then Lemma \ref{lem:normal forms Hum} implies that the equation
\[
u(n)m(n)=a(n)u(n+1)
\]
admits a solution $a(n) \in A$ if and only if 
\[
u(n+1)= d(n)x(n).
\]
This completes the proof of the claim, and of the proposition.
%{\color{red} Notice that the sequence of elements $\wt{b(n)} \in U(c^{-1})$ iteratively constructed through the proof satisfies the following property. For an element $(u(s))_{0 \leq s \leq n} \in U^{n+1}$ such that $u(0)=u_\circ$, we have that
%\[
%(u(s)m(s)u(s+1)^{-1})_{0 \leq s \leq n}  \in A^{n+1} \iff (u(s)m(s)u(s+1)^{-1})_{0 \leq s \leq n}= (\wt{b(s)} \overline{c})_{0 \leq s \leq n}.
%\]
%}
\end{proof}

\medskip\noindent
\emph{Proof of Theorem \ref{thm:cross-sec discrete}}.
Let $m \in M^{\mathbb{Z}}(R)$, and let
\[
m_{\geq n}:=(m(s))_{s \in \mathbb{Z}_{\geq n}} \in M^{\mathbb{Z}_{\geq n}}(R) \qquad (n \in \mathbb{Z}).
\]
By Proposition \ref{prop:truncated Gauge}, there exist elements $u^{(n)} \in U^{\mathbb{Z}_{\geq n}}(R)$ such that $u^{(n)} \cdot m_{\geq n} \in A^{\mathbb{Z}_{\geq n}}(R).$
Let us consider two integers $n_1 \leq n_2$, and set 
\[
u^{(n_1)}_{\geq n_2}:= (u^{(n_1)}(s))_{s \in \mathbb{Z}_{\geq n_2}} \in U^{\mathbb{Z}_{\geq n_2}}(R).
\]
As $u^{(n_1)}_{\geq n_2} \cdot m_{\geq n_2} \in A^{\mathbb{Z}_{\geq n_2}}(R)$, Proposition \ref{prop:truncated Gauge} implies that 
\[
u^{(n_1)}(s)= u^{(n_2)}(s) \qquad \mbox{whenever} \quad s -n_2 \geq \max \{ m_i \mid 1\le i \le r\}.
\]
Hence, the element $u \in U^\mathbb{Z}(R)$ whose components are
\[
u(s):=u^{(n)}(s) \qquad (s,n \in \mathbb{Z}, \ s-n \geq \max \{ m_i \mid 1\le i \le r\}),
\]
is well defined and satisfies $u \cdot m \in A^{\mathbb{Z}}(R).$
Then assume that $v \in U^\mathbb{Z}(R)$ satisfies $v \cdot m \in A^{\mathbb{Z}}(R).$ 
With obvious notation, we have that $v_{\geq n} \cdot m_{\geq n} \in A^{\mathbb{Z}_{\geq n}}(R)$ for every $n \in \mathbb{Z}$. 
Hence, Proposition \ref{prop:truncated Gauge} implies that $v=u.$
\qed

\begin{Cor}
\label{cor:P surjective}
The morphism $P: B(G,c)^\circ \to (L^{c,e})^\mathbb{Z}$ is surjective at the level of $R$-points, for every $\mathbb{C}$-algebra $R$.
\end{Cor}
\begin{proof}
Let $l \in (L^{c,e})^\mathbb{Z}(R)$ and $b(0) \in B^-(R)$. 
Notice that there exist unique elements $b(s) \in B^-(R) \ ( s \in \mathbb{Z} \setminus \{0\}) $ such that 
\[
b(s)b(s+1)^{-1}= l(s), \qquad (s \in \mathbb{Z}). 
\]
Then let $u \in U^\mathbb{Z}(R)$ such that $u \cdot l = a \in A^\mathbb{Z}(R)$ and set 
\[
g(s):= u(s) b(s), \qquad (s \in \mathbb{Z}).
\]
A direct calculation shows that $b:=(g(s))_{s \in \mathbb{Z}}$ belongs to $B(G,c)^\circ(R)$ and $P(b)= l$. 
\end{proof}

\begin{Cor}
\label{cor:H is the quotient}
    The discrete Miura transformation $H: (L^{c,e})^\mathbb{Z} \to A^\mathbb{Z}$ sends a point of $(L^{c,e})^\mathbb{Z}$ to its unique conjugate in $A^\mathbb{Z}$ under the discrete gauge action of $U^\mathbb{Z}$.
\end{Cor}

\begin{proof}
We recall that the discrete Miura transformation $H$ is the unique morphism making the diagram
\[
\begin{tikzcd}
	B(G,c)^\circ && (L^{c,e})^{\mathbb{Z}} \\
	B(G,c) && A^{\mathbb{Z}}
	\arrow["{P}", from=1-1, to=1-3]
	\arrow[ from=1-1, to=2-1]
	\arrow["{H}", from=1-3, to=2-3]
	\arrow["{\pi}", from=2-1, to=2-3]
\end{tikzcd}
\]
commute. 
Here $B(G,c)^\circ \longrightarrow B(G,c)$ is the natural inclusion and the morphism $\pi$ sends a band $(g(s))_{s \in \mathbb{Z}}$ to the sequence $(g(s)g(s+1)^{-1})_{s \in \mathbb{Z}}$.

Let $R$ be a $\mathbb{C}$-algebra and $(g(s))_{s \in \mathbb{Z}} \in B(G,c)^\circ(R)$ be a band such that any component admits a twisted Birkhoff decomposition 
\[
g(s)=u(s)b(s) \qquad  (u(s) \in U(R), \ b(s) \in B^-(R), \ s \in \mathbb{Z}).
\]
Recall that $P(b)=(b(s)b(s+1)^{-1})_{s \in \mathbb{Z}}$.
Equation (\ref{eq:10}) implies that the discrete Gauge action of the sequence $(u(s))_{s \in \mathbb{Z}}$ sends $P(b)$ to $\pi(b)= H(P(b))$.  Then Corollary \ref{cor:P surjective} and Theorem \ref{thm:cross-sec discrete} allow to conclude.
\end{proof}

Summing up, Theorem \ref{thm:cross-sec discrete} implies that the quotient of the action of $U^\mathbb{Z}$ on $M^\mathbb{Z}$ exists and it can naturally be identified with $A^\mathbb{Z}$. 
Then Corollary \ref{cor:H is the quotient} asserts that the discrete Miura transformation $H$ can be identified with the restriction to $(L^{c,e})^\mathbb{Z}$ of the quotient morphism $M^\mathbb{Z} \to A^\mathbb{Z}.$

\bigskip
Finally, observe that the abelian group $\mathbb{Z}$ acts on the schemes $M^\mathbb{Z}$, $U^\mathbb{Z}$ and $A^\mathbb{Z}$ by translating the components of points, and that the isomorphism 
$U^\mathbb{Z} \times A^\mathbb{Z} \to M^\mathbb{Z}$ of Corollary~\ref{cor-34} is equivariant with respect to
these actions. Let $N \ge 2$ and consider the subschemes $M^{(N)}$, $U^{(N)}$, $A^{(N)}$ 
of fixed points under the action of the subgroup $N\mathbb{Z}$. These subschemes can be respectively identified with products of copies of $M$, $U$, $A$ indexed by the cyclic group $\mathbb{Z}/ N \mathbb{Z}$. Moreover the group $U^{(N)}$ acts on $M^{(N)}$ by \emph{cyclic discrete gauge transformation: 
\begin{equation}\label{eq-1.13}
(u(s)) \cdot (m(s)) := \bigl(u(s)m(s)u(s+1)^{-1}\bigr)\qquad (s \in \mathbb{Z}/N\mathbb{Z}). 
\end{equation}
}
Thus by taking fixed points under the action of $N\mathbb{Z}$ in Corollary~\ref{cor-34}, we immediately obtain
the following corollary, which was stated in \cite{SS} :
\begin{Cor}[\cite{SS} Theorem 2.9]
For any $\mathbb{C}$-algebra $R$, the group $U^{\mathbb{Z}/ N \mathbb{Z}}(R)$ acts freely on $M^{\mathbb{Z}/ N \mathbb{Z}}(R)$ 
    by the cyclic discrete gauge transformation (\ref{eq-1.13}), and $A^{\mathbb{Z}/ N \mathbb{Z}}(R)$ is a cross-section. \qed
\end{Cor}
This result yields for $G = SL(2)$ a discrete version of the Virasoro algebra studied in \cite[\S 6]{FRS}, and more
generally it allows to define certain lattice $\mathcal{W}$-algebras \cite[\S 2.4]{SS}. 

\begin{remark}
{\rm
In this section, as in most places in this paper, we have assumed that $G$ is of type $A$, $D$, $E$ in order to 
obtain relations between bands and the representation theory of quantum affine algebras. However,
the cross-section theorem of Steinberg is proved for groups of all types including types $B$, $C$, $F$, $G$, and so is 
Theorem~\ref{thm:cross-sec loop}. % of Semenov-Tian-Shansky and Sevostyanov. 
Our proof of Theorem~\ref{thm:cross-sec discrete} also works more generally
for groups of types $B$, $C$, $F$,~$G$. The only place where we have used the fact that $G$ is of type $A$, $D$, $E$
is the proof of Lemma~\ref{lem:groups Hs}~(1). But this can readily be extended to the general case by replacing the 
Auslander-Reiten quiver of $Q$ by the Auslander-Reiten quiver of a suitable hereditary Artin algebra of finite representation type, see \cite[VIII]{ARS}. 
}
\end{remark}

\subsection{Proof of Proposition~\ref{prop7}}\label{subsec3.11}

We can now give the proof of Proposition~\ref{prop7}.

Let $\delta : L^{c,e} \to (L^{c,e})^\mathbb{Z}$ denote the diagonal embedding,
and let $\gamma := H \circ \delta$ be its composition with the discrete Miura
transformation $H$ :
\[
\xymatrix@-1.0pc{
L^{c,e} \ar[rrdd]^{\gamma}\ar[rr]^{\delta}&&
 (L^{c,e})^\mathbb{Z}\ar[dd]^{H}
\\
\\
&& A^\mathbb{Z} 
}
\]
We claim that $\gamma$ is equal to the composition $\delta' \circ \psi$, where $\delta' : A \to A^\mathbb{Z}$ is the 
diagonal embedding, and $\psi$ is the map defined in \S\ref{subsec:1.2}.
In other words, the following diagram commutes:
\[
\xymatrix@-1.0pc{
L^{c,e} \ar[rrdd]^{\gamma}\ar[rr]^{\delta}\ar[dd]^{\psi}&&
 (L^{c,e})^\mathbb{Z}\ar[dd]^{H}
\\
\\
A\ar[rr]^{\delta'} && A^\mathbb{Z} 
}
\]
By Theorem \ref{thm:cross-sec discrete}, we can consider the morphism $\upsilon: (L^{c,e})^\mathbb{Z} \to U^\mathbb{Z}$ assigning to an element $l$ of $(L^{c,e})^\mathbb{Z}$ the unique $u$ in $U^\mathbb{Z}$ such that $u \cdot l \in A^\mathbb{Z}$.
By Corollary \ref{cor:H is the quotient}, we have 
\begin{equation}
\label{eq:13}
    \upsilon(l) \cdot l= H(l) \qquad(l \in (L^{c,e})^{\mathbb{Z}}).
\end{equation}
Recall that the abelian group $\mathbb{Z}$ acts on the schemes $(L^{c,e})^\mathbb{Z}$, $U^\mathbb{Z}$ and $A^\mathbb{Z}$ by translating the components of points. The invariant elements with respect to these actions are the ones belonging to the diagonals. 
Let $l_\circ \in L^{c,e}$.
Then $\delta(l_\circ)$ is $\mathbb{Z}$-invariant.
Moreover, as the morphisms $\upsilon$ and $H$ are clearly equivariant with respect to the aforementioned $\mathbb{Z}$-actions, $\upsilon (\delta(l_\circ))$ is also $\mathbb{Z}$-invariant. 
Hence, there exists $u_\circ \in U$ such that $\upsilon(\delta(l_\circ))$ is diagonal with every component equal to $u_\circ$.
Then, by Equation (\ref{eq:13}), $H \circ \delta(l_\circ)$ is diagonal with every component equal to $u_\circ l_\circ u_\circ^{-1} \in A$. 
By the definition of the morphism $\psi$, we deduce that $u_\circ l_\circ u_\circ^{-1}=\psi(l_\circ)$. 
This proves the claim.

Let $s \in \mathbb{Z}, \ k \in \mathbb{Z}_{\geq 0}$ and $1\le i \le r$. 
Observe that via the isomorphism $B(G,c)\sslash G \simeq A^\mathbb{Z}$, the function $\theta^{(s)}_{i,k}$ identifies with 
\[
(a(s))_{s \in \mathbb{Z}}\longmapsto\Delta_{\varpi_i, \varpi_i}(a(s)a(s+1) \cdots a(s+k-1)).
\]
Hence, $( \delta')^*(\theta^{(s)}_{i,k})$ is equal to the function $\theta_{i,k}$ introduced in \S\ref{subsec_1.3}.
Because of the claim, we deduce that
\[
(H \circ \delta)^*(\theta^{(s)}_{i,k})= \psi^*(\theta_{i,k})
\]
On the other hand, by Theorem~\ref{Thm-FR}, we have that the pullback under $H$ is identified with the $q$-character homomorphism. 
Recall that to obtain the character of a $U_q(\widehat{\mathfrak{g}})$-module $M$,
we just have to specialize $Y_{j,a} \mapsto y_j$ in the $q$-character of~$M$.
Taking into account the isomorphism $\kappa$ of Theorem~\ref{Thm-FR}, we see 
that $y_j$ corresponds to the function on $L^{c,e}$ sending $l_\circ$ to $\Delta_{w_0(\varpi_j),w_0(\varpi_j)}(l_\circ)$. Hence we deduce that $(H \circ \delta)^*(\theta^{(s)}_{i,k})$ is the character $Q^{(i)}_k$ of the Kirillov-Reshetikhin module $W^{(i)}_{k,\,q^{2s+1-\xi_i}}$
expressed in the variables~$y_j$. \qed

%%%%%%%%%%%%%%%%%%%%%%%%%%%%%%%%%%%%%%%%%%%%%%%%%%%%%%%%
\section{Bands and representations of shifted quantum affine algebras}
%%%%%%%%%%%%%%%%%%%%%%%%%%%%%%%%%%%%%%%%%%%%%%%%%%%%%%%%
\label{sect4}

In \S\ref{sect3} we have constructed cluster algebra structures on the 
invariant subalgebras $R(G,c)^G$, $R(G,c)^U$ and $R(G,c)^{U^-}$ of $R(G,c)$, and we
have explained their connections with the categories $\mathcal{C}_\mathbb{Z}$, $O^+_\mathbb{Z}$
and $O^-_\mathbb{Z}$.
We will now construct a cluster structure on the algebra $R(G,c)$ itself,
and explain its meaning in terms of a category $\mathcal{O}^{\rm shift}_\mathbb{Z}$ of representations
of shifted quantum affine algebras. This was in fact the original motivation for
introducing the scheme $B(G,c)$.

\subsection{Cluster structure on $R(G,c)$}

\begin{figure}[t!]
\[
\def\objectstyle{\scriptscriptstyle}
\def\lablestyle{\scriptscriptstyle}
\xymatrix@-1.0pc{
&\Delta^{(1)}_{c^2\varpi_1,\,\varpi_1}\ar[rd]%\ar[u]
&
&\ar[ld] \Delta^{(1)}_{c^2\varpi_3,\,\varpi_3} %\ar[u]
\\
&&\ar[ld] \Delta^{(0)}_{c^2\varpi_2,\,\varpi_2} \ar[rd]%\ar[uu]
&&
\\
&\ar[uu] \Delta^{(0)}_{c^2\varpi_1,\,\varpi_1} \ar[rd]&
&\ar[ld] \Delta^{(0)}_{c^2\varpi_3,\,\varpi_3} \ar[uu]
\\
&&\ar[uu] {\red \Delta^{(0)}_{c\varpi_2,\,\varpi_2}}\ar[d] &&
\\
&&\ar[ld]\ar[rd] {\green \Delta^{(0)}_{c\varpi_2,\,\wt{c}\varpi_2}} &&
\\
&\ar[uuu]{\red \Delta^{(0)}_{c\varpi_1,\,\varpi_1}}\ar[d]  && {\red \Delta^{(0)}_{c\varpi_3,\,\varpi_3}}\ar[uuu]\ar[d]
\\
&{\green \Delta^{(0)}_{c\varpi_1,\,\wt{c}\varpi_1}} \ar[rd] &&\ar[ld] {\green \Delta^{(0)}_{c\varpi_3,\,\wt{c}\varpi_3}}
\\
&& \ar[d]\ar[uuu]{\red \Delta^{(0)}_{\varpi_2,\,\wt{c}\varpi_2}} &&
\\
&&\ar[ld] {\green \Delta^{(0)}_{\varpi_2,\,\wt{c}^2\varpi_2}} \ar[rd]&&
\\
&\ar[uuu]{\red \Delta^{(0)}_{\varpi_1,\,\wt{c}\varpi_1}}\ar[d]  && {\red \Delta^{(0)}_{\varpi_3,\,\wt{c}\varpi_3}}\ar[d]\ar[uuu]
\\
&{\green \Delta^{(0)}_{\varpi_1,\,\wt{c}^2\varpi_1}} \ar[rd] &&\ar[ld] {\green \Delta^{(0)}_{\varpi_3,\,\wt{c}^2\varpi_3}}
\\
&& \ar[uuu]\Delta^{(-1)}_{\varpi_2,\,\wt{c}^2\varpi_2} \ar[ld]\ar[rd]
&&
\\
&\ar[uu] \Delta^{(-1)}_{\varpi_1,\,\wt{c}^2\varpi_1} &&\ar[uu] \Delta^{(-1)}_{\varpi_3,\,\wt{c}^2\varpi_3} 
\\
}
\]
\caption{\label{Fig3} {\it The labelled quiver $\Gamma$ for $c = s_1s_3s_2$ in type $A_3$.}}
\end{figure} 

We first define the doubly-infinite labelled quiver $\Gamma$.
Its underlying graph can be des\-cri\-bed as a finite modification of the underlying graph
of the seed $\Xi$ of \S\ref{subsec3.7}.
Recall the integers $m_i\ (1\le i\le r)$ introduced in \S\ref{subsect3.4}.
The vertices of $\Xi$ are labelled by pairs $(i,s)\ (1\le i \le r,\ s\in\mathbb{Z})$.
At each vertex $v$ of the finite subset
\[
S:= \{(i,s) \mid 1\le i \le r,\ -m_i\le s \le -1\}
\]
we perform the following local modifications of the graph:
\begin{itemize}
\item[(i)]
For every $v\in S$, replace $v$ by a pair of vertices $v'$ and $v''$ one above the other,
connected by a vertical down arrow $v'\to v''$.
\item[(ii)]
If $v\to w$ is an arrow connecting two vertices $v, w\in S$ in different columns,
replace it by an arrow $v''\to w'$.
\item[(iii)]
If $v\to w$ is an arrow connecting two vertices $v, w\in S$ in the same column,
replace it by an arrow $v'\to w''$.
\item[(iv)]
If $u\to v$ is an arrow connecting a vertex $u \not \in S$ with $v\in S$, 
replace it by an arrow $u\to v'$ if $u$ and $v$ are not in the same column,
 and by an arrow $u\to v''$ if $u$ and $v$ are in the same column.
\item[(v)]
If $v\to u$ is an arrow connecting a vertex $u \not \in S$ with $v\in S$, 
replace it by an arrow $v''\to u$ if $u$ and $v$ are not in the same column, and 
by an arrow $v'\to u$ if $u$ and $v$ are in the same column.
\end{itemize}
As a result, we have added $N$ new vertices, where $N$ is the number of positive roots
of $\mathfrak{g}$, and we have transformed certain oriented 3-cycles of $\Xi$ into oriented 4-cycles.
To highlight these modifications, for every $v\in S$ we will paint the corresponding 
vertex $v'$ in red and $v''$ in green.

All vertices of $\Gamma$ will be labelled by elements of $R(G,c)$ of the form 
\[
\Delta^{(s)}_{c^k(\varpi_i),\,\widetilde{c}^l(\varpi_i)},\qquad (1\le i \le r,\ s\in\mathbb{Z},\ 0\le k,l \le m_i,\ 
m_i-1\le k+l \le m_i),
\]
where $\widetilde{c}:=w_0c^{-1}w_0$ is the Coxeter element dual to $c$.
We refer the reader to \cite[\S5.1]{FL} for the precise labeling rule.

\begin{example}
{\rm
Let $G$ be of type $A_3$ and $c = s_1s_3s_2$.
The corresponding labelled quiver $\Gamma$ is displayed in Figure~\ref{Fig3}.
\qed}
\end{example}

\begin{Thm}[\cite{FL}]\label{Thm33}
The ring $R(G,c)$ of regular functions on the scheme $B(G,c)$ has the structure of a cluster algebra with initial seed
given by $\Gamma$.
\end{Thm}

Consider the finite subset $V^{(0)}$ of vertices of $\Gamma$ labelled by the 
$\Delta^{(s)}_{c^k(\varpi_i),\,\widetilde{c}^l(\varpi_i)}$ for which $s=0$.
This consists of the green and red vertices, together with one rim of black
vertices above the highest red vertices. All these functions 
$\Delta^{(0)}_{c^k(\varpi_i),\,\widetilde{c}^l(\varpi_i)}$
are obtained by pullback of the corresponding generalized minors of $G$
under the natural projection $\pi_0:B(G,c) \to G$ mapping a band $(g(s))_{s\in\mathbb{Z}}$
to its 0th component $g(0)$. 

In fact, the cluster subalgebra of $R(G,c)$ with initial seed given by the 
full subquiver $\Gamma^{(0)}$ of $\Gamma$ supported on $V^{(0)}$ is isomorphic to the cluster algebra
$\mathbb{C}[G]$.
More precisely, let $V_f\subset V^{(0)}$ denote the subset consisting of the $2r$ vertices
sitting on the top and bottom rims of $V^{(0)}$. We regard the cluster variables of $V_f$ as frozen.
A cluster algebra structure on the coordinate ring of the open double Bruhat cell 
$G^{w_0,w_0} := (Bw_0B)\cap (B^-w_0B^-)$ is described in \cite{BFZ}.
One can check that 
$\Gamma^{(0)}$ coincides with the pullback under $\pi_0$ of one of the initial seeds of $\mathbb{C}[G^{w_0,w_0}]$.
However, in \cite{BFZ} the frozen variables are assumed to be invertible, since the 
corresponding minors do not vanish on $G^{w_0,w_0}$. It was recently proved by Oya \cite{O} that if we remove this condition of invertibility, we get a cluster structure on $\mathbb{C}[G]$, (see also \cite{QY}). 
This fact plays an important role in our proof of Theorem~\ref{Thm33}.

Recall the integers
%the product runs over all vertices $j$ of the Dynkin diagram of $\mathfrak{g}$
%connected to vertex $i$, 
$a_{ij} \, (i,j\in I)$ defined in the proof of Proposition~\ref{prop6}. Similarly, we define
$b_{ij} = 1$ if $s_j$ precedes $s_i$ in a reduced decomposition of $\wt{c}$, and otherwise $b_{ij} = 0$.
Let $\Delta^{(0)}_{c^{m_i-1-k}(\varpi_i),\,\wt{c}^k(\varpi_i)}$ be the initial cluster variable
sitting at one of the red vertices. The initial cluster exchange relation at this vertex 
is 
\[
 \Delta^{(0)}_{c^{m_i-1-k}(\varpi_i),\, \wt{c}^{\,k}(\varpi_i)}\Delta^{(0)}_{c^{m_i-k}(\varpi_i),\, \wt{c}^{\,k+1}(\varpi_i)} 
\]
\[
 =
 \Delta^{(0)}_{c^{m_i-k}(\varpi_i),\, \wt{c}^{\,k}(\varpi_i)} \Delta^{(0)}_{c^{m_i-1-k}(\varpi_i),\, \wt{c}^{\,k+1}(\varpi_i)} \\[2mm]
 + 
\prod_{j:\ c_{ij}=-1} \Delta^{(0)}_{c^{m_i-1-k+a_{ij}}(\varpi_i),\, \wt{c}^{\,k+b_{ij}}(\varpi_i)}.
\]
 This identity is the pullback under the morphism $\pi_0$ of another instance of the generalized minor identities of \cite[Theorem 1.17]{FZ}.
There is a similar formula for mutating at green vertices of~$\Gamma$.
This special class of mutations also plays a crucial role in the proof of Theorem~\ref{Thm33}. We will see below its interpretation in terms of shifted quantum affine algebras.

\subsection{$R(G,c)$ and the category $\mathcal{O}^{\rm shift}_\mathbb{Z}$}

In 2019, Finkelberg and Tsimbaliuk \cite{FT} introduced a new class of algebras $U_{q,\mu}(\widehat{\mathfrak{g}})$ depending on a weight $\mu\in P$, called \emph{shifted quantum affine algebras}.
In \cite{H}, a category $\mathcal{O}_\mu$ of $U_{q,\mu}(\widehat{\mathfrak{g}})$-modules was defined and studied, and it was shown that the direct
sum 
\[
\mathcal{O}^{\mathrm{shift}} := \bigoplus_{\mu \in P} \mathcal{O}_\mu
\]
is endowed with a fusion product, which yields a ring structure on its Grothendieck group. The category $\mathcal{O}^{\mathrm{shift}}$ contains the subcategory $\mathcal{C}^{\mathrm{shift}}$ of
finite-dimensional representations. Hernandez has introduced a subcategory $\mathcal{C}^{\mathrm{shift}}_\mathbb{Z}$ of $\mathcal{C}^{\mathrm{shift}}$ and shown that its Grothendieck ring 
%is isomorphic to the Grothendieck ring of $O_\Z^+$ (as based rings), hence 
possesses the same cluster algebra structure as the Grothendieck ring of $O_\mathbb{Z}^+$. 

In \cite{GHL} the subcategory $\mathcal{O}^{\mathrm{shift}}_\mathbb{Z}$ of $\mathcal{O}^{\mathrm{shift}}$ was introduced, and its Grothendieck ring was shown to have a cluster structure.
More precisely, it was shown that the topological ring $K_0(\mathcal{O}^{\mathrm{shift}}_\mathbb{Z})$ contains a cluster algebra $\mathcal{A}_{w_0}$
whose closure equals $K_0(\mathcal{O}^{\mathrm{shift}}_\mathbb{Z})$.

The clusters of a distinguished family of seeds of $\mathcal{A}_{w_0}$ consist of so-called $Q$-variables:
\[
Q_{w(\varpi_i),a},\qquad (1\le i \le r,\ w\in W,\ a\in \mathbb{C}^*).
\]
These are elements of $K_0(\mathcal{O}^{\mathrm{shift}})$ introduced by Frenkel and Hernandez \cite{FH2},
who conjectured that they are classes of simple objects of $\mathcal{O}^{\mathrm{shift}}$.
They showed that they satisfy the following functional relations, called (extended) $QQ$-system.
For $w\in W$ such that $ws_i > w$, 
we have:
\begin{equation}\label{Eq12}
Q_{ws_i(\varpi_i),\,aq} Q_{w(\varpi_i),\,aq^{-1}} = [-w(\alpha_i)] Q_{ws_i(\varpi_i),\,aq^{-1}} Q_{w(\varpi_i),\,aq}
+ \prod_{j:\ c_{ij}=-1} Q_{w(\varpi_j),\,a}.
\end{equation}
In \cite[\S8.2]{GHL}, the following slightly renormalized $Q$-variables in $K_0(\mathcal{O}^{\mathrm{shift}}_\mathbb{Z})$ 
were introduced:
\[
\underline{Q}_{w(\varpi_i),q^{2k-\xi_i}}
,\qquad (1\le i \le r,\ w\in W,\ k\in \mathbb{Z}).
\]
They satisfy a modification of Equation~(\ref{Eq12}) in which the scaling factor $[-w(\alpha_i)]$
disappears. 
Namely, for $w\in W$ such that $ws_i > w$, 
we have:
\[
\underline{Q}_{ws_i(\varpi_i),\,q^{2k-\xi_i}} \underline{Q}_{w(\varpi_i),\,q^{2k-\xi_i-2}} 
\]
\[
=\ \underline{Q}_{ws_i(\varpi_i),\,q^{2k-\xi_i-2}} \underline{Q}_{w(\varpi_i),\,q^{2k-\xi_i}}
+\prod_{j:\ c_{ij}=-1} \underline{Q}_{w(\varpi_j),\,q^{2k-\xi_i-1}}. 
\]

\begin{Prop}[\cite{FL}]
The assignment $\Delta^{(s)}_{c^k(\varpi_i), w(\varpi_i)} \mapsto 
\underline{Q}_{w(\varpi_i),\, q^{2(s+k)-\xi_i}}$ extends to an algebra 
isomorphism from $R(G,c)$ to $\mathbb{C}\otimes {\mathcal A_{w_0}}$ matching the cluster structures
on both sides.
\end{Prop}

In this isomorphism, the exchange relations at green or red vertices of the initial
seed of $R(G,c)$ become instances of $QQ$-system relations. 
As a consequence, one can regard the $QQ$-system relations as emerging from the
generalized minor identities of Fomin and Zelevinsky. This fact was first established by 
Koroteev and Zeitlin \cite{KZ} in their work on $q$-opers and $(G,q)$-Wronskians,
and it has been an important source of inspiration for defining the schemes $B(G,c)$.

\bigskip
{\bf Acknowledgements.}
We thank Edward Frenkel, Christof Geiss, David Hernandez, Masato Okado, and Nicolai Reshetikhin for fruitful exchanges concerning this work.
This work was partially supported by the MUR Excellence Department Project 2023--2027 awarded to the Department of Mathematics, University of Rome Tor Vergata CUP E83C18000100006, and the 
PRIN2022 CUP E53D23005550006.

\bigskip
\small
\noindent
\begin{tabular}{ll}
Luca {\sc Francone} & Università degli studi di Roma Tor Vergata,\\
&Dipartimento di Matematica,\\
&Via della Ricerca Scientifica 1, 00133 Roma \\
& email : {\tt francone@mat.uniroma2.it}
\\[5mm]
Bernard {\sc Leclerc}  & Universit\'e de Caen Normandie,\\
&CNRS UMR 6139 LMNO,\\ &14032 Caen, France\\
&email : {\tt bernard.leclerc@unicaen.fr}
\end{tabular}

\end{document}